\documentclass{emsprocart}
\usepackage{tikz}
\usepackage{amssymb}
\usepackage{makeidx}

\usepackage{amscd,mathrsfs}

\contact[koen.thas@gmail.com]{Koen Thas, Department of Mathematics, Ghent University, Krijgslaan 281 | S25, Ghent, Belgium}

\newcommand{\1}{\mathrm{id}}
\newcommand{\I}{\mathbf{I}}
\newcommand{\K}{\mathbb{K}}
\newcommand{\id}{\mathbf{1}}

\newcommand{\Z}{\mathbb{Z}}
\newcommand{\GL}{\mathbf{GL}}
\newcommand{\F}{\mathbb{F}}
\newcommand{\Fun}{\mathbb{F}_1}

\newcommand{\PG}{\ensuremath{\mathbf{PG}}}
\newcommand{\PGL}{\ensuremath{\mathbf{PGL}}}
\newcommand{\PSL}{\ensuremath{\mathbf{PSL}}}
\newcommand{\Aut}{\ensuremath{\mathrm{Aut}}}

\newcommand{\AG}{\ensuremath{\mathbf{AG}}}

\newcommand{\Spec}{\mathrm{Spec}}

\newcommand{\mD}{\ensuremath{\mathscr{D}}}

\newcommand{\mC}{\ensuremath{\mathscr{C}}}
\newcommand{\mS}{\ensuremath{\mathscr{S}}}

\newcommand{\mF}{\ensuremath{\mathscr{F}}}
\newcommand{\mP}{\ensuremath{\mathscr{P}}}
\newcommand{\mB}{\ensuremath{\mathscr{B}}}
\newcommand{\mQ}{\ensuremath{\mathbf{Q}}}
\newcommand{\mA}{\ensuremath{\mathscr{A}}}

\newcommand{\hB}{\ensuremath{\mathbf{B}}}
\newcommand{\hP}{\ensuremath{\mathbf{P}}}
\def\doubleprod#1#2{\ooalign{$#1\prod$\cr$#1\coprod$\cr}}
\DeclareMathOperator*{\Rprod}{\mathpalette\doubleprod\relax}

\newtheorem{theorem}{Theorem}[section]

\newtheorem{proposition}[theorem]{Proposition}

\theoremstyle{definition}

\newtheorem{remark}[theorem]{Remark}

\newcommand{\wis}[1]{{\text{\em \usefont{OT1}{cmtt}{m}{n} #1}}}
\def\Spec{\wis{Spec}}

\makeindex
\title[The Weyl functor | Introduction to Absolute Arithmetic]{The Weyl functor | Introduction to Absolute\\ Arithmetic}

\author[Koen Thas]{Koen Thas}

\begin{document}
\setcounter{page}{3}

\begin{abstract}
Starting from an ancient observation of Tits concerning the interpretation of symmetric groups as Chevalley groups over 
a (non-existing) field having only one element, we describe combinatorial geometry over this field, as well as Linear Algebra.
We arrive at an ``absolute mantra'' which is one of the basic principles of the present book.
\end{abstract}



\maketitle
\tableofcontents

\newpage
\section{Introduction}

We start this chapter by elaborating on ideas which were hatched from some seminal remarks made by Tits in his early paper
``Sur les analogues alg\'{e}briques des groupes semi-simples complexes'' (1957) \cite{anal}.\\

\medskip
\subsection{Projective $\mathbb{F}_1$-geometry}

When considering a class of incidence geometries which are defined over finite fields | take for instance the class of 
finite classical buildings of a fixed rank and type (we refer to later sections for a formal explanation of these notions) | it sometimes makes
sense to consider the ``limit'' of these geometries when the number of field elements tends to $1$. As a star example, let the class of geometries be the 
classical projective planes $\PG(2,\K)$ defined over finite fields $\K$. Then the number of points per line of such a plane is
\begin{equation}
\vert \K \vert + 1, 
\end{equation}
so in the limit, the ``limit object'' should have $1 + 1$ points incident with every line. On the other hand, we want the limit object still to be an
axiomatic projective plane\index{axiomatic!projective plane}, so we still want it to have the following properties:
\begin{itemize}
\item[(i)]
any two distinct lines meet in precisely one point;
\item[(ii)]
any two distinct points are incident with precisely one line (the dual of (i));
\item[(iii)]
not all points are on one and the same line (to avoid degeneracy).
\end{itemize}

It is clear that such a limit projective plane (``defined over $\mathbb{F}_1$'') should be an ordinary triangle (as a graph). So it is nothing else than a {\em chamber} in the building of any thick projective plane. Note that projective planes precisely are  generalized $3$-gons (which are also to be defined later). \\

Adopting this point of view, it is easily seen that, more generally,
projective $n$-spaces over $\mathbb{F}_1$ should be just sets $X$ of cardinaly $n + 1$ endowed with the geometry of $2^X$: any subset (of cardinality $0 \leq r + 1 \leq n + 1$) is a subspace (of dimension $r$). 
In other words, projective $n$-spaces over $\Fun$ are complete graphs on $n + 1$ vertices with a natural subspace structure.
It is important to note that these spaces still satisfy the Veblen-Young axioms, and that they are the only such incidence geometries with thin lines.

\begin{proposition}[See, e.g., Cohn \cite{Cohn} and Tits \cite{anal}]
\label{propCT}
Let $n \in \mathbb{N} \cup \{-1\}$.
The combinatorial projective space $\PG(n,\mathbb{F}_1) = \PG(n,1)$ is the complete graph on $n + 1$ vertices endowed with the induced geometry of subsets,
and $\Aut(\PG(n,\mathbb{F}_1)) \cong \PGL_{n + 1}(\mathbb{F}_1) \cong \mathbf{S}_{n + 1}$. 
\end{proposition}
\begin{proof}
We already have obtained the geometric part of the proposition. As for the group theoretical part, the symmetric group on $n + 1$ letters clearly is 
the full automorphism group of $\PG(n,1)$. 
\end{proof}

It is extremely important to note that any $\PG(n,\K)$ with $\K$ a division ring contains (many) subgeometries isomorphic to $\PG(n,\Fun)$ as defined above; so the latter object is independent of $\K$, and is the {\em common geometric substructure of all projective spaces of a fixed given dimension}:
\begin{equation}
\underline{\mA}: \{ \PG(n,\K) \vert \K \ \ \mbox{division}\ \ \mbox{ring} \} \longrightarrow \ \ \{ \PG(n,\Fun)\}.
\end{equation}

Further in this chapter, we will formally find the automorphism groups of $\Fun$-vector spaces through matrices, and these groups 
will perfectly agree with Proposition \ref{propCT}. We will also investigate other examples of limit buildings, as first described by Tits in \cite{anal}. In fact, we will look for a more general functor $\underline{\mA}$ (called {\em Weyl functor}\index{Weyl!functor} for reasons to be explained later) from a certain category of more general incidence geometries than buildings, to its subcategory of fixed objects under $\underline{\mA}$.  \\

Note that over $\mathbb{F}_1$,  

\begin{equation}\mathbf{P\Gamma L}_{n + 1}(\mathbb{F}_1)  \cong \PGL_{n + 1}(\mathbb{F}_1)  \cong \PSL_{n + 1}(\mathbb{F}_1).\end{equation}

\medskip
\subsection{Counting functions}

It is easy to see the symmetric group also directly as a limit with $\vert \K \vert \longrightarrow 1$ of linear groups $\PG(n,\K)$ (with the dimension fixed).
The number of elements in $\PG(n,\K)$ (where $\K = \F_q$ is assumed to be finite and $q$ is a prime power) is
\begin{equation}
\frac{(q^{n + 1} - 1)(q^{n + 1} - q)\cdots(q^{n + 1} - q^n)}{(q - 1)} = (q - 1)^nN(q)
\end{equation}
for some polynomial $N(X) \in \Z[X]$, and we have 
\begin{equation}
N(1) = (n + 1)! = \vert \mathbf{S}_{n +  1} \vert.\\
\end{equation}

\medskip
Now let $n,q \in \mathbb{N}$, and define $[n]_q = 1 + q + \cdots + q^{n - 1}$. (For $q$ a prime power, $[n]_q = \vert \PG(n,q)\vert$.) Put $[0]_q! = 1$, and define 
\begin{equation}
[n]_q! := [1]_q[2]_q\ldots [n]_q.
\end{equation}

Let $\mathbf{R}$ be a ring, and let $x, y, q$ be ``variables'' for which $yx = qxy$. Then there are polynomials $\left[\begin{array}{c}n \\ k\end{array}\right]_q$ in $q$ with integer coefficients,
such that
\begin{equation}
(x + y)^n = \sum_{k = 0}^n\left[\begin{array}{c}n \\ k\end{array}\right]_qx^ky^{n - k}.
\end{equation}

Then 
\begin{equation}
\left[\begin{array}{c}n \\ k\end{array}\right]_q = \frac{[n]_q!}{[k]_q![n - k]_q!},
\end{equation}
and if $q$ is a prime power, this is the number of $(k - 1)$-dimensional subspaces of $\PG(n - 1,q)$ ($=\wis{Grass}(k,n)(\F_q)$). The next proposition again gives sense to the limit 
situation of $q$ tending to $1$.

\begin{proposition}[See e.g. Cohn \cite{Cohn}]
The number of $k$-dimensional linear subspaces of $\PG(n,\mathbb{F}_1)$, with $k \leq n \in \mathbb{N}$, equals

\begin{equation}
\left[\begin{array}{c}n + 1 \\ k + 1\end{array}\right]_1 = \frac{n!}{(n - k)!k!} = \left[\begin{array}{c}n + 1 \\ k + 1\end{array}\right].
\end{equation}
\end{proposition}

Many other enumerative formulas in Linear Algebra, Projective Geometry, etc. over finite fields $\F_q$ seem to keep meaningful inerpretations if $q$
tends to $1$, and this phenomenon (the various interpretations) suggests a deeper theory in characteristic one.

\medskip
\subsection{The Weyl functor}

In this chapter, we will consider various categories $\mathbf{C}$ of combinatorial objects, and in a first stage these objects will come 
with certain field data (later we will also consider categories were no obvious field data are available). We will look for a functor
$\underline{\mA}$ which associates with the objects $o$ of $\mathbf{C}$ interpretations of $o$ over the field with one element, $\Fun$, keeping in mind
that $\Fun$ does not exist, but $\underline{\mA}(o)$ {\em does}. In all those categories, expressions of the form
\begin{equation}
\underline{\mA}(o)\ \ +\ \ \mbox{field}\ \mbox{data}
\end{equation}
make sense, in that the knowledge of $\underline{\mA}(o)$ together with field data will single out uniquely defined objects in the $\underline{\mA}$-fiber of $o$. In principle, many objects in $\mathbf{C}$ could descend to some $\underline{\mA}(o)$, but with additional field data, we can point to a unique object. (Think, for instance, of the category $\mathbf{C}$ of projective spaces over finite fields  with natural morphisms; applying $\underline{\mA}$ to $o = \PG(n,\F_q)$ yields the aforementioned geometry $\PG(n,1)$ which is independent of $\F_q$, so the $\underline{\mA}$-fiber consists of all finite $n$-dimensional projective spaces. But giving the additional data of a single field yields a unique projective space coordinatized by this field.) So the functor $\underline{\mA}$ comes with a number of base extension arrows to fields, and together with these arrows, the originial theories can be reconstructed from below.

Since we will consider many different categories $\mathbf{C}$, we want $\underline{\mA}$ to  be defined in such a way that it commutes with
various natural functors between these categories, an  example of this prinicple being the diagram
\begin{equation}
\begin{array}{ccc}
\{ \PG(n,\F_q) \vert n, q \} &\overset{\underline{\mA}}\longrightarrow &\{ \PG(n,\Fun) \vert n \}\\
&&\\
\downarrow& &\downarrow\\
&&\\
\{ \PGL_{n + 1}(q) \vert n,q\} &\overset{\underline{\mA}}\longrightarrow &\{\mathbf{S}_{n + 1} \vert n \}\\
&&\\
\end{array}
\end{equation}
which we already considered.

 \section{BN-Pairs and the Weyl functor}

Before introducing the general concepts of building and BN-pair, we study the standard example of projective spaces (from the building point of view).\\

\subsection{Projective space}

Let $\mathbf{R}$ be a division ring (= skew field), let $n \in \mathbb{N}$, and let $V = V(n,\mathbf{R}) = \mathbf{R}^n$\index{$V(n,\mathbf{R})$} be the $n$-dimensional (left or right) vector space over $\mathbf{R}$. We define the {\em $(n - 1)$-dimensional (left or right) projective space}\index{projective!space} $\PG(n,\mathbf{R})$\index{$\PG(n,\mathbf{R})$} as being the set
\begin{equation}
(\mathbf{R}^n \setminus \{0\})/\sim,
\end{equation}
where the equivalence relation ``$\sim$'' is defined by (left or right) proportionality, with the subspace structure being induced by that of $V$. (When $n$ is not finite, similar definitions hold.) The choice of ``left'' or ``right'' does not affect the isomorphism class.
If $\mathbf{R} = \F_q$ is the finite field with $q$ elements ($q$ a prime power), we also write $\PG(n - 1,q)$\index{$\PG(n - 1,q)$} instead of 
$\PG(n - 1,\F_q)$. Sometimes the notations $\hP^{n - 1}(\mathbf{R})$\index{$\hP^{n - 1}(\mathbf{R})$}, $\hP^{n - 1}(q)$\index{$\hP^{n - 1}(q)$}, $\mathbb{P}^{n - 1}(\mathbf{R})$\index{$\mathbb{P}^{n - 1}(\mathbf{R})$} and $\mathbb{P}^{n - 1}(q)$\index{$\mathbb{P}^{n - 1}(q)$} occur as well.

There is also a notion of {\em axiomatic projective space}\index{axiomatic!projective space}, which is defined to be an {\em incidence geometry} (defined later in this section) which is governed by certain axioms, which are (of course) satisfied by ``classical'' projective spaces over division rings. A truly remarkable thing is that Veblen and Young \cite{VY} showed that if the dimension $n - 1$ of such a space is at least three, it {\em is} isomorphic to some $\PG(n - 1,\mathbf{R})$. And this is well-known {\em not} to be true when the dimension is less than three. \\

\subsection{Representing spaces as group coset geometries}

Let $\hP$ be a finite\\ 
($n$-)dimensional projective space over some division ring $\mathbf{R}$. 
Consider any $\mathbf{R}$-base $\hB$. Define a simplicial complex (in the next section to be formally defined, and called ``apartment''\index{apartment}) $\mC \equiv \mC(\hB)$, by letting it be the set of all possible subspaces of $\hP$ generated by 
subsets of $\hB$. (Let it also contain the empty set.) Define a (maximal) ``flag''\index{flag} or {\em chamber}\index{chamber} in $\mC$ as a maximal chain (so of length $n + 1$) of subspaces
in $\mC$. Let $F$ be such a fixed flag.

Consider the special projective linear group
$K := \PSL_{n + 1}(\mathbf{R})$ of $\mathbf{P}$. 
Then note that $K$ acts transitively on the pairs $(\mC(\hB'),F')$,
where $\hB'$ is any $\mathbf{R}$-base and $F'$ is a maximal flag in $\mC(\hB')$.

Put $B = K_{\mC}$ and $N = K_{F}$; then note the following properties:

\begin{itemize}
\item[(i)] 
$\langle B,N \rangle = K$;
\item[(ii)]
put $H = B \cap N
\lhd N$ and $N/H = W$;
then $W$ obviously is isomorphic to the symmetric group $\mathbf{S}_{n + 1}$ on $n + 1$ elements. Note that a presentation of $\mathbf{S}_{n + 1}$ is:
\begin{equation}
\langle   s_i \vert s_i^2 = \1, (s_is_{i + 1})^3 = \1, (s_is_j)^2 = \1,     i, j \in \{1,\ldots,n + 1\}, j \ne i \pm 1          \rangle.
\end{equation}
\item[(iii)]
$Bs_iBwB \subseteq BwB \cup
Bs_iwB$ whenever $w \in W$ and $i\in\{1,2,\ldots,n + 1\}$;
\item[(iv)]
$s_iBs_i \ne B$ for all $i\in\{1,2,\ldots,n + 1\}$.
\end{itemize}

(Here, expressions such as $BwB$ mean $B\widetilde{w}HB$, where $\widetilde{w}$ is any representant of $\widetilde{w}H = w$.)

Now let $K \cong \PSL_{n + 1}(\mathbf{R})$ be as above, and suppose that $B$ and $N$ are groups satisfying these properties. Define a geometry $\mB_{(K,B,N)}$\index{$\mB_{(K,B,N)}$} as follows. 

\begin{itemize}
\item[(B1)]
Its elements are left cosets in $K$ of the groups $P_i$ which properly contain $B$ and are different from $K$, $i = 1,\ldots,n + 1$;
\item[(B2)]
two elements $gP_i$ and $hP_j$ are incident if they intersect nontrivially.\\
\end{itemize}

\begin{proposition}
$\mB_{(K,B,N)}$ is isomorphic to $\PG(n,\mathbf{R})$.
\end{proposition}

\medskip
\subsubsection{Low dimensional cases}

For dimension $n = 1$, our definition of axiomatic space doesn't make much sense. 
Here we rather {\em start} from a division ring $\mathbf{R}$, and define $\hP$, the {\em projective line} over $\mathbf{R}$, as being the set $(\mathbf{R}^2 \setminus \{0\})/\sim$,
where $\sim$ is defined by (left) proportionality. So we can write

\begin{equation}
\hP = \{(0,1)\} \cup \{(1,\ell) \vert \ell \in \mathbf{R}\}.
\end{equation}
Now $\PSL_2(\mathbf{R})$ acts naturally on $\mathbf{P}$; in fact, we have defined the projective line as a permutation group equipped with the natural doubly transitive action
of $\PSL_2(\mathbf{R})$.
Defining a geometry as we did for higher rank projective spaces,
through the ``$(B,N)$-pair structure''
of $\PSL_2(\mathbf{R})$, one obtains the same notion of projective line.\\

Restricting to finite fields, we obtain the following very simple

\begin{proposition}
A finite projective line has $q + 1$ points, for some prime power $q$.
\end{proposition}

The $2$-dimensional case is different, still. Here, other than in the $1$-dimensional case, one obtains a nontrivial geometry; the axioms now boil down to just demanding that each two different points are incident with precisely one line, that, dually, any two distinct lines intersect in precisely one point, and there exist four distinct points of which no three are on the same line. So we need not require additional algebraic structure in order to have interesting objects. Here we cannot say much about the order of the plane a priori.

\subsubsection{Representation by diagram}

We represent the presentation of $\mathbf{S}_{n + 1}$ as above
in the following way (this will be explained in more detail in the next section):

\bigskip
$\mathbf{A}_{n + 1}$: \begin{tikzpicture}[style=thick, scale=1.5]
\foreach \x in {1,2,3,5,6}{
\fill (\x,0) circle (2pt);}

\draw (1,0) -- (2,0);
\draw (2,0) -- (3,0);
\draw (3,0) -- (3.5,0);
\draw (4.5,0) -- (5,0);
\draw (5,0) -- (6,0);
\draw (4,0) node {$\dots$} ;

\end{tikzpicture} \hspace{0.5cm} ($n\geq 0$)\\

(The number of vertices is $n + 1$ | each vertex corresponding to an
involution in the generating set of involutions | and
we have an edge between vertices $s_i$ and $s_j$
if and only if $\vert j-i\vert=1$.)\\

\subsection{Simplicial complexes} 
 
 Recall that a (combinatorial) {\em simplicial complex}\index{simplicial complex} is a pair $(\mS,Y)$, where $Y$ is a set and $\mS \subseteq 2^Y$, such that $Y \in \mS$ and 
\begin{equation}
U \subseteq V \in \mS \Longrightarrow U \in \mS.
\end{equation}
We are ready to introduce buildings. We will not provide each result with a specific reference | rather, we refer the reader to \cite{AbBr}.\\

\subsection{Combinatorial definition}

A {\em chamber geometry}\index{chamber!geometry} is a geometry 
\begin{equation}
\Gamma = (\mC_1,\mC_2,\ldots,\mC_j,\I)
\end{equation}
 of rank $j$  (so $\Gamma$ has $j$ different
kinds of elements and $\I$ is an incidence relation between the elements such that no two elements belonging to the
same $\mC_i$, $1 \leq i \leq j$, can be incident) so that the simplicial complex $(\mC,X)$, where $\mC = \cup_{i = 1}^{j}\mC_i$
and $S \subseteq \mC$ is contained in $X$ if and only if every two distinct elements of $S$ are incident, is a chamber
complex (as in, e.g., \cite{POL}).
A {\em building}\index{building} $(\mC,X)$ is a thick chamber geometry $(\mC_1,\mC_2,\ldots,\mC_j,\I)$
of rank $j$, where $\mC = \cup_{i = 1}^j\mC_i$,
together with a set $\mA$ of thin chamber subgeometries, so that:

\begin{itemize}
\item[(i)]
every two chambers are contained in some element of $\mA$;
\item[(ii)]
for every two elements $\Sigma$ and $\Sigma'$ of $\mA$ and every two simplices $F$ and $F'$, respectively contained in $\Sigma$ and
$\Sigma'$, there exists an isomorphism $\Sigma \mapsto \Sigma'$ which fixes all elements of both $F$ and $F'$.
\end{itemize}

If all elements of $\mA$ are finite, then the building is called {\em spherical}\index{spherical!building}. Elements of $\mA$ are called {\em apartments}\index{apartment}. \\

\subsection{Coxeter groups and systems}

We need to introduce the notions of ``Coxeter system'' and ``Coxeter diagram''.\\

\subsubsection{Coxeter groups}
A {\em Coxeter group}\index{Coxeter!group} is a group with a presentation of type 

\begin{equation}
\langle s_1,s_2,\ldots,s_n \vert (s_is_j)^{m_{ij}} = \1 \rangle,
\end{equation}

\noindent
where $m_{ii} = 1$ for all $i$, $m_{ij} \geq 2$ for $i \ne j$, and $i, j$ are natural numbers bounded above by the natural number $n$. If $m_{ij} = \infty$, no relation of the form
$(s_is_j)^{m_{ij}}$ is imposed. All generators in this presentation are involutions. The natural number $n$ is the {\em rank}\index{rank} of the Coxeter group.

A {\em Coxeter system}\index{Coxeter!system} is a pair $(W,S)$, where $W$ is a Coxeter group and $S$ the set of generators defined by the presentation.

Recall that a {\em dihedral group}\index{dihedral group} of {\em rank}\index{rank} $n$, denoted $\mathbf{D}_n$\index{$\mathbf{D}_n$}, is the symmetry group of a regular $n$-gon in the real plane. \\

\subsubsection{Coxeter matrices}

A square $n \times n$-matrix $(M)_{ij}$ is a {\em Coxeter matrix}\index{Coxeter!matrix} if it is symmetric and defined over $\mathbb{Z} \cup \{\infty\}$, has only $1$s on the diagonal, and has $m_{ij} \geq 2$ if $i \ne j$. 
Starting from a Coxeter matrix $(M)_{ij}$, one can define a Coxeter group $\langle s_1,s_2,\ldots,s_n \vert (s_is_j)^{m_{ij}} = \1 \rangle$, and conversely.
Different Coxeter systems can give rise to the same Coxeter group, even if the rank is different. \\

\subsubsection{Coxeter diagrams}

Let $(W,S)$ be a Coxeter system. Define a graph, called ``Coxeter diagram''\index{Coxeter!diagram}, as follows. Its vertices are the elements of $S$. If $m_{ij} = 3$, we draw a single edge between $s_i$ and $s_j$; if $m_{ij} = 4$, a double edge, and if $m_{ij} \geq 5$, we draw a single edge with label $m_{ij}$. If $m_{ij} = 2$, nothing is drawn.
If the Coxeter diagram is connected, we call $(W,S)$ {\em irreducible}\index{irreducible!Coxeter diagram}. If it has a finite number of vertices, we call $(W,S)$ {\em spherical}\index{spherical!Coxeter diagram}.\\

The irreducible, spherical Coxeter diagrams were classified by H. S. M. Coxeter \cite{HSMC}; the complete list is the following.

\bigskip
$\mathbf{A}_n$\index{Coxeter!diagram!of type $\mathbf{A}_n$}: 
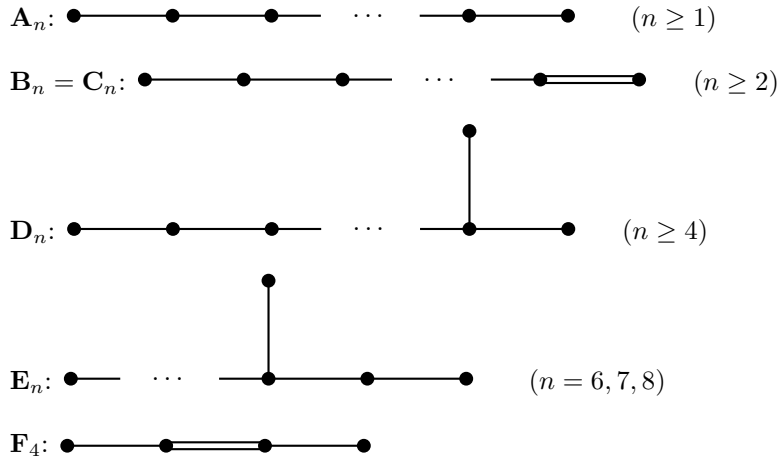
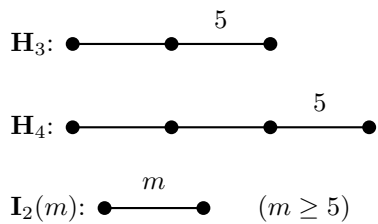
\begin{tikzpicture}[style=thick, scale=1.3]
\foreach \x in {1,2,3,5,6}{
\fill (\x,0) circle (2pt);}

\draw (1,0) -- (2,0);
\draw (2,0) -- (3,0);
\draw (3,0) -- (3.5,0);
\draw (4.5,0) -- (5,0);
\draw (5,0) -- (6,0);
\draw (4,0) node {$\dots$} ;

\end{tikzpicture} \hspace{0.5cm} ($n\geq 1$)\\

$\mathbf{B}_n = \mathbf{C}_n$\index{Coxeter!diagram!of type $\mathbf{C}_n$}\index{Coxeter!diagram!of type $\mathbf{B}_n$}: \begin{tikzpicture}[style=thick, scale=1.3]
\foreach \x in {1,2,3,5,6}{
\fill (\x,0) circle (2pt);}

\draw (1,0) -- (2,0);
\draw (2,0) -- (3,0);
\draw (3,0) -- (3.5,0);
\draw (4.5,0) -- (5,0);
\draw (5,0.035) -- (6,0.035);
\draw (5,-0.035) -- (6,-0.035);
\draw (4,0) node {$\dots$} ;

\end{tikzpicture}\hspace{0.5cm} ($n\geq 2$)\\

$\mathbf{D}_n$\index{Coxeter!diagram!of type $\mathbf{D}_n$}: \begin{tikzpicture}[style=thick, scale=1.3]
\foreach \x in {1,2,3,5,6}{
\fill (\x,0) circle (2pt);}

\fill (5,1) circle (2pt);

\draw (1,0) -- (2,0);
\draw (2,0) -- (3,0);
\draw (3,0) -- (3.5,0);
\draw (4.5,0) -- (5,0);
\draw (5,0) -- (6,0);
\draw (5,0) -- (5,1);
\draw (4,0) node {$\dots$} ;

\end{tikzpicture}\hspace{0.5cm}  ($n\geq 4$)\\

$\mathbf{E}_n$\index{Coxeter!diagram!of type $\mathbf{E}_n$}: \begin{tikzpicture}[style=thick, scale=1.3]
\foreach \x in {0,2,3,4}{
\fill (\x,0) circle (2pt);}

\fill (2,1) circle (2pt);

\draw (0,0) -- (.5,0);
\draw (1.5,0) -- (2,0);
\draw (2,0) -- (3,0);
\draw (3,0) -- (4,0);
\draw (2,0) -- (2,1);
\draw (1,0) node {$\dots$} ;

\end{tikzpicture}
\hspace{0.5cm}  ($n=6,7,8$)\\

$\mathbf{F}_4$\index{Coxeter!diagram!of type $\mathbf{F}_4$}: \begin{tikzpicture}[style=thick, scale=1.3]
\foreach \x in {0,1,2,3}{
\fill (\x,0) circle (2pt);}

\draw (0,0) -- (1,0);
\draw (2,0) -- (3,0);
\draw (1,0.035) -- (2,0.035);
\draw (1,-0.035) -- (2,-0.035);

\end{tikzpicture}\\

$\mathbf{H}_3$\index{Coxeter!diagram!of type $\mathbf{H}_3$}: \begin{tikzpicture}[style=thick, scale=1.3]
\foreach \x in {0,1,2}{
\fill (\x,0) circle (2pt);}

\draw (0,0) -- (1,0);
\draw (1,0) -- (2,0);
\draw (1.5,.25) node {$5$} ;

\end{tikzpicture}\\

$\mathbf{H}_4$\index{Coxeter!diagram!of type $\mathbf{H}_4$}: \begin{tikzpicture}[style=thick, scale=1.3]
\foreach \x in {-1,0,1,2}{
\fill (\x,0) circle (2pt);}
\draw (-1,0) -- (0,0);
\draw (0,0) -- (1,0);
\draw (1,0) -- (2,0);
\draw (1.5,.25) node {$5$} ;

\end{tikzpicture}\\

$\mathbf{I}_2(m)$\index{Coxeter!diagram!of type $\mathbf{I}_2(m)$}: \begin{tikzpicture}[style=thick, scale=1.3]
\foreach \x in {1,2}{
\fill (\x,0) circle (2pt);}

\draw (1,0) -- (2,0);
\draw (1.5,.25) node {$m$} ;

\end{tikzpicture}\hspace{0.5cm}  ($m\geq 5$)\\

\medskip
\subsection{Incidence geometries}

 Having met some standard examples of incidence geometries, we now 
 introduce general incidence geometries\index{incidence!geometry} in a formal way. These objects will be our way to approach 
 combinatorial geometries in the present chapter.

 An {\em incidence geometry}\index{incidence!geometry} or  {\em Buekenhout-Tits geometry}\index{Buekenhout-Tits!geometry} consists of a set $X$ of {\em objects}\index{objects}
provided with a symmetric relation $\I$ called {\em incidence}\index{incidence}
and a surjective function 

\begin{equation}
t:X \longrightarrow I 
\end{equation}
that assigns a {\em type}
to each object, such that two objects of the same type are never incident.
The set $I$ is the set of {\em types}\index{type}. The cardinality $|I|$
is called the {\em rank}\index{rank!of incidence geometry} of the geometry.

Denote the geometry by $\Gamma = \Gamma(X,\I,I,t)$.

If the rank is two, we also speak of a {\em point-line geometry}\index{point-line geometry} (the assignment
\begin{equation}
\I\ \ \longrightarrow\ \  \{\mbox{point},\mbox{line} \}
\end{equation}
being bijective but arbitrary).

\subsubsection{Geometries as incidence graphs}

It will be particularly important in the $\mathbb{F}_1$-context
 to see incidence geometries as a kind of graph. An extra feature which comes in handy is that essentially (or better: usually) the automorphism group $A$ of the 
 geometry is the same as the automorphism group $B$ of the associated graph. In any case, $A \leq B$, and if 
 $[B : A] \ne 1$, then this quantity is a measure for the number of types of objects that play the same role.\\

An incidence geometry can be viewed as a multipartite graph $\Gamma$
with vertex set $X$ and partition $\{X_i \mid i \in I\}$
(with $X_i = t^{-1}(i)$), with incidence taken as adjacency.\\

The geometry $\Gamma$ is called {\em connected}\index{connected incidence geometry} when the graph $\Gamma$
is connected. The graph without vertices is not connected:
a connected graph has precisely one connected component, while the graph
without vertices has no connected component.

A {\em flag}\index{flag} $F$ in $\Gamma$ is a clique, which is by definition a complete subgraph. No two elements of a flag
have the same type. The {\em rank}\index{rank!of flag} of $F$ is $|t(F)|$ (that is, $|F|$).
The {\em corank}\index{corank of a flag} of $F$ is $| I \setminus t(F)|$.
The {\em residue}\index{residue} ${\rm Res}(F)$ (also written $\Gamma_F$) is the geometry
with set of objects $Y = \{y \in X \setminus F \mid F \cup \{y\}
~\mbox{\rm is a flag}\}$, incidence inherited from $\Gamma$,
set of types $I \setminus t(F)$, and type function inherited from $\Gamma$.
The geometry $\Gamma$ is called {\em residually connected}\index{residually connected} when every residue
of rank at least two is connected (and hence nonempty), and every residue
of rank one is nonempty.\\

\medskip
\subsection{BN-Pairs and buildings}

A group $G$ is said to have
a {\em BN-pair}\index{BN-pair} $(B,N)$, where $B, N$ are subgroups of $G$, if:

\begin{itemize}
\item[(BN1)] 
$\langle B,N \rangle = G$;
\item[(BN2)] 
$H = B \cap N
\lhd N$ and $N/H = W$ is a Coxeter group with distinct generator set of involutions
$S = \{ s_j \vert j \in J \}$;
\item[(BN3)] 
$BsBwB \subseteq BwB \cup
BswB$ whenever $w \in W$ and $s\in S$;
\item[(BN4)] 
$sBs \ne B$ for all $s\in S$.
\end{itemize}

The group $B$, respectively $W$, is a {\em Borel subgroup}\index{Borel subgroup},
respectively the {\em Weyl group}\index{Weyl!group}, of $G$. The  quantity $\vert S\vert$
is called the \emph{rank}\index{rank!of BN-pair} of the BN-pair. If $W$ is a finite group, the BN-pair is {\em spherical}\index{spherical!BN-pair}. It is {\em irreducible}\index{irreducible!BN-pair} if the corresponding Coxeter system is.
Sometimes we call $(G,B,N)$ also a {\em Tits system}\index{Tits system}.\\

\begin{remark}
{\rm
Asking that $W$ is a Coxeter group is in fact redundant; by the other axioms and the fact that $S$ consists of involutions, it is not hard to show that $W$ {\em must be} a Coxeter group, and that $S$ is uniquely determined as the set of elements in $W^\times$ for which
\begin{equation}
B \cup BsB
\end{equation}
is a group.
}
\end{remark}

\subsubsection{Buildings as group coset geometries}

To each Tits system $(G,B,N)$ one can associate a building $\mB_{(G,B,N)}$ in a natural way, through a group coset construction. For that reason we introduce the standard {\em parabolic subgroups}\index{parabolic subgroup}; these are just the proper subgroups of $G$ which properly contain $B$. Let $I \subset J$, and define 
\begin{equation}
W_I := \langle s_i \vert i \in I \rangle \leq W.
\end{equation}

Then 
\begin{equation}
P_I := BW_IB
\end{equation}
is a subgroup of $G$ which obviously contains $B$, and vice versa it can be shown that any standard parabolic subgroup has this form.

We are ready to introduce $\mB_{(G,B,N)}$.

\begin{itemize}
\item[(B1)]
\textsc{Elements}: (or ``subspaces'' or ``varieties'') are elements of the left coset spaces $G/P_I$, $\emptyset \ne I \subset J \ne I$.
\item[(B2)]
\textsc{Incidence}: $gP_I$ is incident with $hP_L$, $I \ne L$, if these cosets intersect nontrivially.
\end{itemize}

The {\em rank}\index{rank!of a building} of $\mB_{(G,B,N)}$ is the rank of the BN-pair.
The building $\mB_{G,B,N}$ is {\em spherical}\index{spherical!building} when the BN-pair $(B,N)$ is; note that this is in accordance with the aforementioned synthetic definition of ``spherical building'' (taken that
there is already a BN-pair around). It is {\em irreducible}\index{irreducible!building} when $(B,N)$ is irreducible.\\

\subsubsection{$G$ as an automorphism group}

The group $G$ acts as an automorphism group, by multiplication on the left, on $\mB_{(G,B,N)}$.  The kernel  $K$ of this action is the biggest normal subgroup of $G$ contained in $B$, and is equal to
\begin{equation}
K=\bigcap_{g\in G}B^g.
\end{equation}

For the sake of convenience, suppose $J$ is the finite set $\{ 1,2,\ldots,n\}$, $n \in \mathbb{N} \setminus \{0\}$.
The group $G/K$ acts faithfully on $\mB_{(G,B,N)}$ and the stabilizer of the flag 
\begin{equation}
F = \{P_{\{1\}},P_{\{1,2\}},\ldots,P_J\}
\end{equation}
is $B/K$. If $K = \{\1\}$, we say that the Tits system is {\em effective}\index{effective Tits system}.\\

Let $\Sigma$ be an apartment of $\mB_{(G,B,N)}$, and let its elementwise stabilizer be $E$; then $NE$ is the global
stabilizer of $\Sigma$. 
We can write

\begin{equation}
E = \bigcap_{w\in W}B^w.
\end{equation}

The next theorem sums up several properties.

\begin{theorem}[\cite{AbBr,Titslect}] \label{Titssystem} Let $(G,B,N)$ be a Tits system with Weyl group $W$. Then the geometry $\mB_{(G,B,N)}$
is a Tits building. Setting
\begin{equation}
K=\bigcap_{g\in G}B^g\mbox{ and } E=\bigcap_{w\in W}B^w, \end{equation}
we have that $G/K$
acts naturally and faithfully by left translation on
$\mB_{(G,B,N)}$. Also, $B$ is the stabilizer of a unique flag $F$
and $NE$ is the stabilizer of a unique apartment containing $F$,
and the triple $(G/K,B/K,NE/K)$ is a Tits system
associated with $\mB_{(G,B,N)}$. 
Moreover, $G/K$ acts transitively on the sets $(A,F')$, where $A$ is an apartment and $F'$ is a maximal flag (chamber) in $A$.
\end{theorem}

The Tits system $(G,B,N)$ is called
\emph{saturated}\index{saturated} precisely when
$N=NE$, with $E$ as above. Replacing $N$ by $NE$, every Tits
system is ``equivalent'' to a saturated one.\\

\subsubsection{Bruhat decomposition}

Let $G$ be a group with a spherical, saturated, effective BN-pair $(B,N)$. Then the ``Bruhat decomposition''\index{Bruhat decomposition} tells us that 

\begin{equation}
G = BWB = \coprod_{w \in W}BwB,
\end{equation}
where $W = N/(B \cap N)$ is the Weyl group. Note that with $I \subset J$, we also have 

\begin{equation}
P_I = BW_IB = \coprod_{w \in W_I}BwB.
\end{equation}

\subsubsection{Classification of BN-pairs}

 If the rank of an abstract spherical building is at least $3$, Tits showed in a celebrated work
 \cite{Titslect} that it is always associated to a BN-pair in the way explained above, and this deep observation led him eventually to classify all spherical BN-pairs of rank $\geq 3$ (cf. \cite[11.7]{Titslect}).
 
So Tits realized a far reaching generalization of the Veblen-Young theorem for spherical buildings, which roughly could be formulated as follows.

\begin{theorem}[Classification of spherical buildings | Tits \cite{Titslect}]
An
irreducible spherical building of rank at least $3$ arises from a simple algebraic group (of relative rank at least $3$) over an arbitrary division ring.\\
\end{theorem}

\bigskip
\begin{tabular}{c | c}
 & \\
\mbox{\textsc{Projective}}\ \ \mbox{\textsc{spaces}} &\mbox{\textsc{Buildings}}\\
 & \\
\hline \\
\mbox{\textsc{Veblen-Young:}}   & \mbox{\textsc{Tits:}} \\
 & \\
\mbox{dim} $\geq 3$:  & \mbox{rank} $\geq 3$: \\
\mbox{vector}\ \ \mbox{spaces} &\mbox{BN-pairs;}\ \ \mbox{simple}\ \ \mbox{algebraic} \ \ \mbox{groups} \\
\mbox{over} \ \ \mbox{division} \ \ \mbox{rings}   & \mbox{over} \ \ \mbox{division} \ \ \mbox{rings}\\
 & \\
 \mbox{dim} $2$:  & \mbox{rank} $2$: \\
 \mbox{axiomatic}\ \ \mbox{projective}\ \ spaces &\mbox{generalized}\ \ \mbox{polygons} \\
  & \\
   & \\
\end{tabular}

\medskip
\subsection{The rank $2$ case}

Combinatorially, a  {\em generalized $n$-gon}\index{generalized!$n$-gon} ($n \geq 3$) is a point-line geometry $\Gamma = (\mP,\mB,\I)$
for which  the following axioms are satisfied:

\begin{itemize}
\item[(i)]
$\Gamma$ contains no ordinary $k$-gon (as a subgeometry), for $2 \leq k < n$;
\item[(ii)]
any two elements $x,y \in \mP \cup \mB$ are contained in some ordinary $n$-gon (as a subgeometry) in
$\Gamma$;
\item[(iii)]
there exists an ordinary $(n + 1)$-gon (as a subgeometry) in $\Gamma$.
\end{itemize}

A {\em generalized polygon}\index{generalized!polygon} (GP)\index{GP} is a generalized $n$-gon for some $n$. 

By (iii), generalized polygons have at least three points per line and three lines per point.  The generalized $3$-gons are precisely
the projective planes. A geometry $\Gamma$ which satisfies (i) and (ii) is a {\em weak generalized $n$-gon}\index{weak generalized $n$-gon}.
If (iii) is not satisfied for $\Gamma$, then $\Gamma$ is called {\em thin}\index{thin}.
Otherwise, it is called {\em thick}\index{thick}. Sometimes we will speak of ``thick (respectively thin)
generalized $n$-gon'' instead of ``thick (respectively thin) weak generalized $n$-gon''. 
Generalized polygons were introduced by Tits in his triality paper \cite{Tits};  the basic reference is \cite{POL}.\\

It is not hard to show that once (iii) is also satisfied for a weak generalized polygon, there are constants $s$ and $t$ such that 
each point is incident with $t + 1$ lines, and each line is incident with $s + 1$ points. If the polygons were to be classical (that is, {\em Moufang}\index{Moufang} \cite{POL}), then there are division rings $\K$ and $\mathbb{L}$ such that $s + 1 = \vert \K \vert + 1$ and $t + 1 = \vert \mathbb{L} \vert + 1$.
If (iii) is not satisfied, it can be shown that each line is incident with precisely $1 + 1$ points, so that thin generalized polygons are the polygons over $\Fun$. 
And thin generalized $n$-gons are nothing else than ordinary $n$-gons. We will come back to this in more detail.

Note that there are many equivalent definitions for the notion of generalized polygon\footnote{Think, e.g., of the definition of generalized $3$-gon and of axiomatic projective plane given in the introduction to this chapter.}, but the present one is very natural in the characteristic one context | we only use 
$\Fun$-polygons to describe the axioms.\\

The relation between buildings and generalized polygons,
as observed by Tits in \cite{Titslect} (see also \cite{POL}, \S 1.3.7, of which the notation is used), is as follows:

\begin{itemize}
\item[(S)]
{\em Suppose $(\mC,X)$, $\mC = \mC_1 \cup \mC_2$, is a spherical building of rank $2$. Then $\Gamma = (\mC_1,\mC_2,\I)$ is a
generalized polygon. Conversely, suppose that $\Gamma = (\mP,\mB,\I)$ is a generalized polygon, and let $\mF$ be the set of its flags.
Then $(\mP \cup \mB, \emptyset \cup \{ \{v\} \vert v \in \mP \cup \mB \} \cup \mF)$ is a chamber geometry of rank $2$.
Declaring the thin chamber geometry corresponding to any ordinary subpolygon an apartment, we obtain a spherical
building of rank $2$.}
\end{itemize}

\medskip
\subsubsection{Duality}

Interchanging the role of points and lines, that is, applying the map
\begin{equation}
\Gamma = (\mP,\mB,\I)\ \  \overset{D}\longrightarrow\ \  \Gamma^D = (\mB,\mP,\I),
\end{equation}
we obtain the (point-line) {\em dual}\index{dual} of $\Gamma$. It is also a GP (with the parameters switched).

\medskip
\subsubsection{Polygons as graphs}

Let $\mS = (\mP,\mB,\I)$ be a generalized $n$-gon. The {\em (point-line) incidence graph}\index{incidence!graph} $(V,E)$ of $\mS$ is defined by taking $V = \mP \cup \mB$, where an edge is drawn between vertices if the corresponding elements in $\mS$ are incident; $(V,E)$ then is a bipartite graph of diameter $n$ and girth $2n$. Vice versa, such graphs 
define GPs. 

Let the graph corresponding to $\mS$ be denoted by $\Gamma$.
We 
call $(x_0,\ldots,x_k)$ a {\em (simple) path}\index{path} if the $x_i$ are pairwise distinct and $x_i$ is adjacent to $x_{i+1}$ 
for $i = 0,\ldots,k - 1$. The natural graph theoretic distance function on $\Gamma$ is denoted 
by ``$\mathrm{d}$'' or sometimes ``$\mathrm{d}_n$''\index{$\mathrm{d}$}\index{$\mathrm{d}_n$}. The set of elements at distance $i$ from some element $x \in \Gamma$ is 
denoted by $\Gamma_i(x)$\index{$\Gamma_i(x)$}. Elements at distance $n$ are called {\em opposite}\index{opposite}.

\subsection{$\mathbb{F}_1$-Buildings and the Weyl functor}

We now define the Weyl functor, and
describe some of the examples Tits mentioned in \cite{anal}.\\

\subsubsection{The Weyl functor\index{Weyl!functor}}

Note again that since $\mathbb{F}_1$ expresses the idea of an Absolute Arithmetic, it is clear that the buildings of a certain prescribed type $\mathbf{T}$ over $\mathbb{F}_1$ should be present in any thick building of the same type $\mathbf{T}$.\\

Motivated by the properties which a building over $\mathbb{F}_1$ of type $\mathbf{T}$ should have, we are ready to define such geometries (and their groups) in general.
Let  $\mB = (\mC_1,\mC_2,\ldots,\mC_j,\I)$ be a thick building
of rank $j$ and type $\mathbf{T}$ (given by one of the Coxeter diagrams below), and let $\mA$ be its set of thin chamber subgeometries.
Suppose $(B,N)$ is a saturated effective BN-pair associated to $\mB$; its Weyl group $W$ is a Coxeter group defined by one of the Coxeter graphs below.

\begin{proposition}
A building of rank $j$ and type $\mathbf{T}$ defined over $\mathbb{F}_1$ is isomorphic
to any element of $\mA$. Its automorphism group is isomorphic to the  Coxeter group $W$.
\end{proposition}

\bigskip
$\mathbf{A}_n$: 
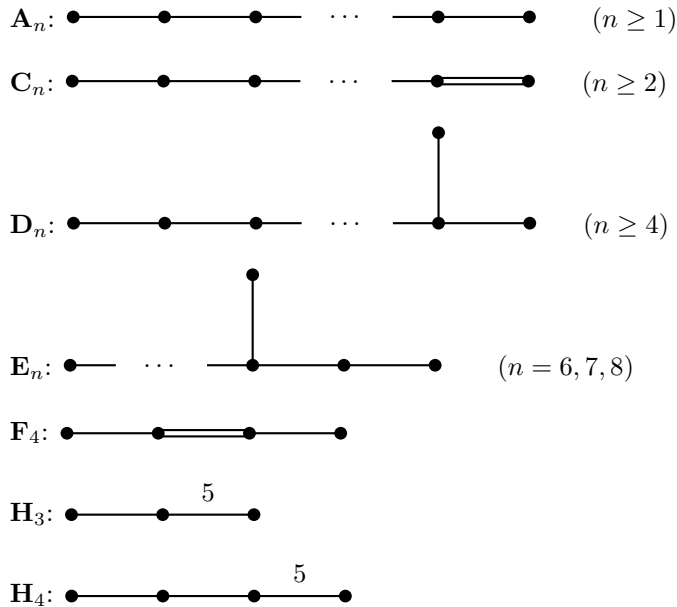
\begin{tikzpicture}[style=thick, scale=1.2]
\foreach \x in {1,2,3,5,6}{
\fill (\x,0) circle (2pt);}

\draw (1,0) -- (2,0);
\draw (2,0) -- (3,0);
\draw (3,0) -- (3.5,0);
\draw (4.5,0) -- (5,0);
\draw (5,0) -- (6,0);
\draw (4,0) node {$\dots$} ;

\end{tikzpicture} \hspace{0.5cm} ($n\geq 1$)\\

$\mathbf{C}_n$: \begin{tikzpicture}[style=thick, scale=1.2]
\foreach \x in {1,2,3,5,6}{
\fill (\x,0) circle (2pt);}

\draw (1,0) -- (2,0);
\draw (2,0) -- (3,0);
\draw (3,0) -- (3.5,0);
\draw (4.5,0) -- (5,0);
\draw (5,0.035) -- (6,0.035);
\draw (5,-0.035) -- (6,-0.035);
\draw (4,0) node {$\dots$} ;

\end{tikzpicture}\hspace{0.5cm} ($n\geq 2$)\\

$\mathbf{D}_n$: \begin{tikzpicture}[style=thick, scale=1.2]
\foreach \x in {1,2,3,5,6}{
\fill (\x,0) circle (2pt);}

\fill (5,1) circle (2pt);

\draw (1,0) -- (2,0);
\draw (2,0) -- (3,0);
\draw (3,0) -- (3.5,0);
\draw (4.5,0) -- (5,0);
\draw (5,0) -- (6,0);
\draw (5,0) -- (5,1);
\draw (4,0) node {$\dots$} ;

\end{tikzpicture}\hspace{0.5cm}  ($n\geq 4$)\\

$\mathbf{E}_n$: \begin{tikzpicture}[style=thick, scale=1.2]
\foreach \x in {0,2,3,4}{
\fill (\x,0) circle (2pt);}

\fill (2,1) circle (2pt);

\draw (0,0) -- (.5,0);
\draw (1.5,0) -- (2,0);
\draw (2,0) -- (3,0);
\draw (3,0) -- (4,0);
\draw (2,0) -- (2,1);
\draw (1,0) node {$\dots$} ;

\end{tikzpicture}
\hspace{0.5cm}  ($n=6,7,8$)\\

$\mathbf{F}_4$: \begin{tikzpicture}[style=thick, scale=1.2]
\foreach \x in {0,1,2,3}{
\fill (\x,0) circle (2pt);}

\draw (0,0) -- (1,0);
\draw (2,0) -- (3,0);
\draw (1,0.035) -- (2,0.035);
\draw (1,-0.035) -- (2,-0.035);

\end{tikzpicture}\\

$\mathbf{H}_3$: \begin{tikzpicture}[style=thick, scale=1.2]
\foreach \x in {0,1,2}{
\fill (\x,0) circle (2pt);}

\draw (0,0) -- (1,0);
\draw (1,0) -- (2,0);
\draw (1.5,.25) node {$5$} ;

\end{tikzpicture}\\

$\mathbf{H}_4$: \begin{tikzpicture}[style=thick, scale=1.2]
\foreach \x in {-1,0,1,2}{
\fill (\x,0) circle (2pt);}
\draw (-1,0) -- (0,0);
\draw (0,0) -- (1,0);
\draw (1,0) -- (2,0);
\draw (1.5,.25) node {$5$} ;

\end{tikzpicture}\\

$\mathbf{I}_2(m)$: \begin{tikzpicture}[style=thick, scale=1.2]
\foreach \x in {1,2}{
\fill (\x,0) circle (2pt);}

\draw (1,0) -- (2,0);
\draw (1.5,.25) node {$m$} ;

\end{tikzpicture}\hspace{0.5cm}  ($m\geq 5$)\\

\subsubsection{Rank $2$ case}

Generalized $n$-gons over $\mathbb{F}_1$ are ordinary $n$-gons, and their automorphism groups are dihedral groups $\mathbf{D}_n$ \cite{anal}. 
It follows that the corner stones of the spherical buildings of rank at least $3$ over $\mathbb{F}_1$ are the ordinary $n$-gons with $n = 3, 4, 6, 8$ (since these gonalities are the only ones which do occur in the corresponding thick buildings by \cite{FeHi}). Still, it is important to note that in the rank $2$ examples, {\em all} positive integer values for $n$ occur (except $n = 0,1,2$).

\subsubsection{Quadrics}

We give one final explicit example | it concerns quadrics. \\

Let $n \in \mathbb{N}^\times$.
A {\em quadric}\index{quadric over $\F_1$} of projective dimension $2n$ or $2n + 1$ over $\mathbb{F}_1$ is a set $\mQ$ of $2(n + 1)$ points arranged in pairs $x_0,y_0,x_1,y_1,\ldots,x_n,y_n$, and its subspaces are 
the subsets not containing any couple $(x_i,y_i)$. The Witt index of the so defined quadrics is $n$. The quadrics in dimension $2n$ have the further property that the maximal singular subspaces ($n$-spaces consisting of $n + 1$ points) are partitioned in two types, namely those containing an even number of points of $\{a_0,a_1,\ldots,a_n\}$ and those containing an odd number. Automorphisms are permutations of the set $\mQ$ which preserve the given pairing in the $2n$-dimensional case.\\

\bigskip
\subsubsection{Trees as $\mathbb{F}_1$-geometries}

If we allow the value $\infty = n$ in the definition of generalized $n$-gon, we obtain a point-line geometry $\Gamma$ without closed paths, such that any two points or lines are contained in a path  without end points.
So $\Gamma$ becomes a tree (allowing more than $2$ points per line) without end points. Its apartments are paths without end points, and the Weyl group is an infinite dihedral group (generated by the reflections about two different adjacent vertices of such an apartment). So in this setting, a generalized $\infty$-gon over $\mathbb{F}_1$ is a tree of valency $2$  without end points. 

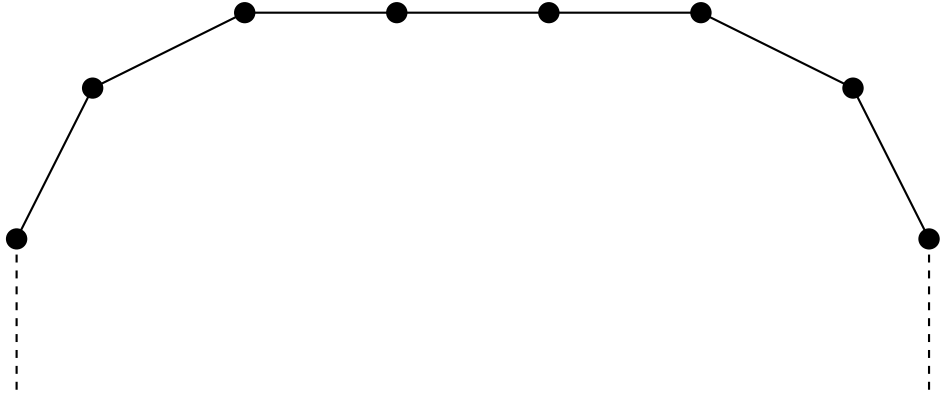
\begin{figure}
\begin{center}
\begin{tikzpicture}[style=thick, scale=2]
\foreach \x in {1,2,3,4}{
\fill (\x,0) circle (2pt);}
{\fill (0,-0.5) circle (2pt);}
{\fill (5,-0.5) circle (2pt);}
{\fill (-0.5,-1.5) circle (2pt);}
{\fill (5.5,-1.5) circle (2pt);}
\draw (1,0) -- (2,0);
\draw (2,0) -- (3,0);
\draw (3,0) -- (4,0);
\draw (4,0) -- (5,-0.5);
\draw (0,-0.5) -- (-0.5,-1.5);
\draw (1,0) -- (0,-0.5);
\draw (5,-0.5) -- (5.5,-1.5);
\draw[dashed] (-0.5,-2.5) -- (-0.5,-1.5);
\draw[dashed] (5.5,-2.5) -- (5.5,-1.5);

\end{tikzpicture} 
\caption{Generalized $\infty$-gon over $\mathbb{F}_1$.}
\end{center}
\end{figure}

Consider, for instance, $G = \mathbf{SL}_2(\mathbb{F}_q((t^{-1})))$. Then $G$ has a BN-pair $(B,N)$, where 

\begin{equation}
B = \{  \left(\begin{array}{cc} a & b\\ c & d   \end{array}\right) \in  \mathbf{SL}_2(\mathbb{F}_q[[t^{-1}]])  \vert c \equiv 0\mod{t^{-1}}\},
\end{equation}
and $N$ is the subgroup of $G$ consisting of elements with only $0$ on the diagonal or only $0$ on the antidiagonal.
Its Weyl group is an infinite dihedral group generated by 

\begin{equation}
s_1 =  \left(\begin{array}{cc} 0 & -1\\ 1 & 0   \end{array}\right)\ \ \mathrm{and}\ \ s_2 =  \left(\begin{array}{cc} 0 & -t\\ 1/t & 0   \end{array}\right).
\end{equation}

The corresponding building (defined in the same way as before) is a generalized $\infty$-gon with $q + 1$ points per line and $q + 1$ lines per point.
Its apartments are exactly the trees we introduced earlier in this section.\\

 \bigskip
\subsection{Generalizations}

One notes that it is possible to relax the BN-pair axioms and still get a meaningful theory, and a Weyl functor. 
For example, let us fix a certain category of groups $\mathbf{C}$, and let us
consider groups $G$ with subgroups $C, E$ such that

\begin{itemize}
\item[(GBN1)] 
$\langle C,E \rangle = G$;
\item[(GBN2)] 
$H = C \cap E
\lhd E$ and $E/H$ is isomorphic to an object of $\mathbf{C}$;
\item[(GBN3)] 
(to be filled in appropriately).
\end{itemize}

Now we fix a set of groups $\mC$ with the property that each of its elements is a proper subgroup of $G$ which properly contains $C$, and
construct a geometry $\Gamma = \Gamma(C,E;\mathbf{C},\mC)$ as follows.

\begin{itemize}
\item[(GB1)]
\textsc{Elements}: are elements of the left coset spaces $G/P$, $P \in \mC$.
\item[(GB2)]
\textsc{Incidence}: $gP$ is incident with $hP'$, $P \ne P'$, if these cosets intersect nontrivially.
\end{itemize}

With $\mathbf{C}$ = category of Coxeter groups and $\mC$ = all proper subgroups properly containing $C$, and (GBN3) replaced by 
the Bruhat decomposition axioms, we obtain a BN-pair $(C,E)$.

Defining (maximal) flags as before, the reader notes the following:

\begin{proposition}
$G$ acts by left translation as an automorphism group of $\Gamma$. This action is transitive on the maximal flags of the geometry.
\end{proposition}

The Weyl functor for this category of group data would be 

\begin{equation}
\underline{\mA}: G \longrightarrow E/H,
\end{equation}
and on the geometrical side, it should send the geometry to the geometry induced by the left coset spaces $E/H$ ($= \{eH \vert e \in E\}$).\\

\quad\textsc{Question}\quad
{\em Find good candidates for $\mathbf{C}$ and $\mC$, and
formulations for {\rm (GBN3)} such that the Weyl images have precisely $2$ points per line.}\\

It is precisely this kind of question which we will be considering in the next section from the synthetic side. (The game that we play there is first
to imagine what the Weyl images should be | certain thin geometries which are fixed objects of the functor we want to define | and build, using 
certain prescribed axioms, the general geometries ``defined over $\Fun$'' using the Weyl geometries as bricks. One could play the same game here:
imagine what the geometries induced on $E/H$ should be, and ask a similar question.)\\

\section{Synthetic geometry over $\mathbb{F}_1$}

In this section, we consider good axioms for incidence geometries to be naturally {\em defined over} $\mathbb{F}_1$.  
This has already been done in various ways for schemes, and in the next chapters we will be concerned by this matter. Still, apart from
 some remarks made by Tits in his 1957 paper, not much seems to be known prior to our paper \cite{NotesI} (on which the present section is based). We want to distinguish between geometries defined over $\mathbb{F}_1$ (or 
 {\em $\mathbb{F}_1$-geometries}\index{$\Fun$-geometry}) and their {\em $\mathbb{F}_1$-versions}. Let $\mC$ be a class of incidence geometries (say, of Buekenhout-Tits geometries with some 
 prescribed set of axioms, cf. below). If $\mC$ (that is, all its elements) will be defined over $\mathbb{F}_1$, we want to have a (Weyl) functor at our disposal which maps any element
 of $\mC$ to its ``$\mathbb{F}_1$-version''; this will be a possibly degenerate incidence geometry which also satisfies the aforementioned axioms, and it will be independent of the chosen element in $\mC$. 
 A model example is 
 the class of generalized $m$-gons with $m \in \mathbb{N} \setminus \{0,1\}$ (that is, the rank $2$ spherical buildings); they will all be defined over $\mathbb{F}_1$ (whether or not they are themselves defined over a ``real'' field), and the images under the functor we seek to define are  ordinary $m$-gons. 
 
 The situation we want to describe can be best (and even almost precisely) compared to the principle of base extension/descent in scheme theory.
 In fact, in the second chapter of the author in this volume we will show that once this theory has been established, there will be an analogy between $\mathbb{F}_1$-incidence geometry and $\mathbb{F}_1$-scheme theory which goes much further than one would suspect at first.  (The interplay between both theories enables one to study, for instance,  large classes of groups (including Chevalley groups) as automorphism groups of schemes over $\mathbb{F}_1$.) 
 
 The details can be found in \cite{NotesI}.  \\

 \subsection{Incidence geometries related to diagrams}

In this chapter we will consider incidence geometries {\em related to diagrams}.  
An axiom system is then imposed by providing a {\em Buekenhout-Tits
diagram}\index{Buekenhout-Tits!diagram} as explained in the next paragraph.\\

\subsection{Buekenhout-Tits diagrams}

Let $\mathscr{D}$ be a labeled graph on $I$, where for $i,j \in I$
the label $\mathscr{D}_{ij}$ is a class of rank $2$ geometries.
We say that $\mathscr{D}$ is a {\em Buekenhout-Tits diagram}\index{Buekenhout-Tits!diagram} for
the geometry $\Gamma = (X,\I,I,t)$
when for every flag $F$ of $\Gamma$ of corank 2,
say $t(F) = I \setminus \{ i,j\}$, the residue $\Gamma_F$
belongs to the class of geometries $\mathscr{D}_{ij}$.

This is a recursive definition for the concept of diagram
in terms of what the labeled edges mean for rank 2 geometries.\\

\subsection{Some traditional labels}

We introduce the nomenclature for some frequently used labels.

\bigskip
\begin{tikzpicture}[style=thick, scale=1.3]
\foreach \x in {1,2}{
\fill (\x,0) circle (2pt);}

\end{tikzpicture} \hspace{0.5cm} $\mathbf{Di}$\index{$\mathbf{Di}$-label}:
 Every $i$-object is incident with every $j$-object.

\bigskip
\begin{tikzpicture}[style=thick, scale=1.3]
\foreach \x in {1,2}{
\fill (\x,0) circle (2pt);}

\draw (1,0) -- (2,0);

\end{tikzpicture}\hspace{0.5cm} $\mathbf{A}_2$\index{$\mathbf{A}_2$-label}: The $i$-objects and $j$-objects form the
points and lines of an axiomatic projective plane.

\begin{tikzpicture}[style=thick, scale=1.3]
\foreach \x in {1,2}{
\fill (\x,0) circle (2pt);}

\draw (1,0.035) -- (2,0.035);
\draw (1,-0.035) -- (2,-0.035);

\end{tikzpicture}\hspace{0.5cm} $\mathbf{B}_2$\index{$\mathbf{B}_2$-label}: The $i$-objects and $j$-objects form the
points and lines of a generalized quadrangle.

\begin{tikzpicture}[style=thick, scale=1.3]
\foreach \x in {1,2}{
\fill (\x,0) circle (2pt);}

\draw (1,0) -- (2,0);
\draw (1.5,.25) node {$m$} ;

\end{tikzpicture}\hspace{0.5cm} $\mathbf{I}_2(m)$ ($m\in \{6,8\}$)\index{$\mathbf{I}_2(m)$-label}: The $i$-objects and $j$-objects form the
points and lines of a generalized hexagon/octagon.

\begin{tikzpicture}[style=thick, scale=1.3]
\foreach \x in {1,2}{
\fill (\x,0) circle (2pt);}

\draw (1,0) -- (2,0);
\draw (1.5,.25) node {$\mathbf{Af}$} ;

\end{tikzpicture}\hspace{0.5cm} $\mathbf{Af}$\index{$\mathbf{Af}$-label}:
The $i$-objects and $j$-objects form the
points and lines of an axiomatic affine plane.

 \begin{tikzpicture}[style=thick, scale=1.3]
\foreach \x in {1,2}{
\fill (\x,0) circle (2pt);}

\draw (1,0) -- (2,0);
\draw (1.5,.25) node {$\mathbf{C}$} ;

\end{tikzpicture}\hspace{0.5cm} $\mathbf{C}$\index{$\mathbf{C}$-label}: The $i$-objects and $j$-objects form the
points and edges of a complete graph.\\

Many other such diagrams are used, but those will be of no concern for our purposes.\\

\subsection{An example}

The geometry of points, lines and planes in a $3$-dimen\-sio\-nal combinatorial
projective space satisfies the axioms given by the diagram

\begin{tikzpicture}[style=thick, scale=1.3]
\foreach \x in {1,2,3}{
\fill (\x,0) circle (2pt);}

\draw (1,0) -- (2,0);
\draw (2,0) -- (3,0);

\end{tikzpicture}

\bigskip
By the result of Veblen-Young, combinatorial $3$-dimensional projective spaces 
necessarily  come from (left or right) vector spaces over a skew field. It is an easy exercise to prove the Veblen-Young
axiom from the $\mathbf{A}_n$-diagram, so that the following holds.

\begin{theorem}
A (thick) Buekenhout-Tits geometry satisfying
the $\mathbf{A}_n$ diagram axioms is a projective space.
\end{theorem}

So the axioms which are imposed on the Buekenhout-Tits geometry by the $\mathbf{A}_n$-diagram suffice to fully determine the incidence geometry.

\subsection{The general Weyl functor}

For the case of buildings, we have seen that the natural way to associate to a building its $\mathbb{F}_1$-building/version,
is through the functor\index{$\underline{\mathscr{A}}$}

\begin{equation}
\underline{\mathscr{A}}: \mathbb{B} \rightarrow \mathbb{A},  
\end{equation}
from the category of (spherical) buildings to the category of apartments of such buildings\index{$\mathbb{A}$}\index{$\mathbb{B}$}. Let us use the same notation for the more 
general hypothetical functor which associates to a geometry (satisfying strong enough axioms), its ``$\mathbb{F}_1$-component'', and let us also keep the notation $\mathbb{A}$ for the more general version of Weyl geometries we are seeking. We want to see the objects in $\mathbb{A}$ also as objects 
of $\mathbb{B}$.\\

The $\mathbb{F}_1$-functor $\underline{\mathscr{A}}$ should have several properties (with respect to the images):

\begin{itemize}
\item[A1|]
all lines should have at most  $2$ different points;
\item[A2|]
an image should be a ``universal object'', in the sense that it should be a subgeometry of any thick geometry of the same ``type'' (defined over {\em any} field, if at all defined over one) of at least the same rank (as we will later see, it will correspond to scheme theoretic base descent to $\mathbb{F}_1$); 
\item[A3|]
it should, of course, still carry the same axiomatic structure (so that $o \in \mathbb{A}$ and elements of $\underline{\mA}^{-1}(o)$ carry the same Buekenhout-Tits diagram);
\item[N|]
it should give a geometric meaning to (certain) arithmetic formulas which express (certain) combinatorial properties of the finite thick geometries we want to define, assumed
to have $s + 1$ points incident with every line, evaluated at the value $s = 1$;
\item[F|]
as $\mathbb{A}$ will be a subclass of the class of $\mathbb{F}_1$-geometries, it should consist precisely of the fixed elements of $\underline{\mathscr{A}}$.
 \end{itemize}
 
 \begin{proposition}[Conjecturally, \cite{NotesI}]
 \label{propoconj}
 Consider $\underline{\mA}: \mathbb{B} \longrightarrow \mathbb{A}$. Then $\mathbb{A}$ is given by the solutions of
 \begin{equation}
 \underline{\mA}(X) = (X).
 \end{equation}
 (The functor ``retracts'' $\mathbb{B}$ to $\mathbb{A}$.)
 \end{proposition}
 
 \begin{remark}
 {\rm Contrary to the base extension theory we will later speculate on, not every incidence geometry is suited to be defined over $\mathbb{F}_1$. (In general, 
 without imposing extra structure on such a geometry, examples are too wild.)
 }
 \end{remark}
 
 Some other remarks need to be made.
 
 \begin{itemize}
\item[A1$'$|]
We work up to point-line duality: that is why we are allowed to ask, without loss of generality, that lines have at most two points. We do {\em not} ask that they
have {\em precisely}  two points, one motivation being e.g. (combinatorial) affine spaces over $\mathbb{F}_1$, in which any line has precisely one point. (And their scheme theoretic versions have precisely one {\em closed point}\index{closed!point}, cf. later chapters for a formal definition.)
\item[CL|]
Referring to the preceding remark, we already note that later on, $\mathbb{F}_1$-geometries with precisely $2$ points per line will correspond to {\em closed subschemes}\index{closed!subscheme} (cf. later chapters) of the appropriate ambient projective $\mathbb{F}_1$-space (as a scheme). (If they contain lines with only one point, one will need to invoke open sets to define the natural associated $\mathbb{F}_1$-scheme.)
\item[A4|]
In some sense, the number of lines through a point of an element $\Gamma$ of $\mathbb{A}$ should reflect the {\em rank} of the geometries in $\underline{\mathscr{A}}^{-1}(\Gamma)$. Think for example of the combinatorial affine and projective spaces over $\mathbb{F}_1$, and the ``Weyl geometries'' of buildings as described by Tits.
Note that this is not a feature of incidence geometries in general, but it appears to be a property which is encoded in the $\underline{\mathscr{A}}$-image of an incidence geometry.
 \end{itemize}

Our natural starting point in \cite{NotesI} was the category of Buekenhout-Tits geometries. The reader is noted that all buildings are members (and so all Chevalley group schemes are automorphism group schemes of members). 
We only consider connected geometries | the general theory can be reduced to the connected theory as usual. We call this assumption ``C''.

The first step is to classify the elements of $\mathbb{A}$. We take A1-A2-A3-A4-C to be the main axioms. After having determined $\mathbb{A}$ \cite{NotesI}, one
defines the functor $\underline{\mathscr{A}}$, and the inverse image $\underline{\mathscr{A}}^{-1}(\mathbb{A})$ in $\mathbb{BT}$\index{$\mathbb{BT}$}, the category of Buekenhout-Tits geometries with obvious morphisms. This inverse image is denoted by $\mathbb{BT}_{\vert 1}$\index{$\mathbb{BT}_{\vert 1}$}.

We refer the reader to \cite{NotesI} for more details. Let us just mention the rank one and two examples in $\mathbb{A}$.

\subsection{Rank $1$ | $\mathbb{A}$}

\bigskip
\begin{tikzpicture}[style=thick, scale=1.3]
\foreach \x in {1,2}{
\fill (\x,0) circle (2pt);}

\end{tikzpicture} \hspace{0.5cm}  $\mathbf{Di}_1$:
 Every $i$-object is incident with every $j$-object. Over $\mathbb{F}_1$, this example has one line and one point, and they are incident.

\bigskip
\begin{tikzpicture}[style=thick, scale=1.3]
\foreach \x in {1,2}{
\fill (\x,0) circle (2pt);}

\draw (1,0) -- (2,0);

\end{tikzpicture}\hspace{0.5cm} $\mathbf{A}_1$: The $i$-objects and $j$-objects form the
points and lines of a combinatorial projective line over $\mathbb{F}_1$: two distinct points incident with a line.

\begin{tikzpicture}[style=thick, scale=1.3]
\foreach \x in {1,2}{
\fill (\x,0) circle (2pt);}

\draw (1,0) -- (2,0);
\draw (1.5,.25) node {$\mathbf{Af}$} ;

\end{tikzpicture}\hspace{0.5cm} $\mathbf{Af}$:
The $i$-objects and $j$-objects form the
points and lines of a combinatorial affine line over $\mathbb{F}_1$: one point incident with one line (the ``absolute flag''\index{absolute!flag}).

\bigskip
\subsection{Rank $2$ | $\mathbb{A}$}

The rank $2$ examples of Buekenhout-Tits geometries are the most important ones, since all other examples (ignoring the rank $0$ and $1$ cases) are 
constructed from these from axioms governed by the diagrams. By A4, any point is incident with at most two lines. Taking this property into account, the reader easily sees that the geometries must be of one of the following types
(where at the end, we introduce a new type).\\

\bigskip
\begin{tikzpicture}[style=thick, scale=1.3]
\foreach \x in {1,2}{
\fill (\x,0) circle (2pt);}

\end{tikzpicture} \hspace{0.5cm}  $\mathbf{Di}_2$:
 Every $i$-object is incident with every $j$-object. Over $\mathbb{F}_1$, this example has two lines and two points, and any point is incident with any line.

\bigskip
\begin{tikzpicture}[style=thick, scale=1.3]
\foreach \x in {1,2}{
\fill (\x,0) circle (2pt);}

\draw (1,0) -- (2,0);

\end{tikzpicture}\hspace{0.5cm} $\mathbf{A}_2$: The $i$-objects and $j$-objects form the
points and lines of a combinatorial projective plane over $\mathbb{F}_1$: an ordinary triangle.

\begin{tikzpicture}[style=thick, scale=1.3]
\foreach \x in {1,2}{
\fill (\x,0) circle (2pt);}

\draw (1,0.035) -- (2,0.035);
\draw (1,-0.035) -- (2,-0.035);

\end{tikzpicture}\hspace{0.5cm} $\mathbf{B}_2$: The $i$-objects and $j$-objects form the
points and lines of a generalized quadrangle, which is an ordinary $4$-gon.

\begin{tikzpicture}[style=thick, scale=1.3]
\foreach \x in {1,2}{
\fill (\x,0) circle (2pt);}

\draw (1,0) -- (2,0);
\draw (1.5,.25) node {$m$} ;

\end{tikzpicture}\hspace{0.5cm} $\mathbf{I}_2(m)$ ($m\in \mathbb{N} \cup \{\infty\}$, $m \geq 5$): The $i$-objects and $j$-objects form the
points and lines of an ordinary $m$-gon.

\begin{tikzpicture}[style=thick, scale=1.3]
\foreach \x in {1,2}{
\fill (\x,0) circle (2pt);}

\draw (1,0) -- (2,0);
\draw (1.5,.25) node {$\mathbf{Af}$} ;

\end{tikzpicture}\hspace{0.5cm} $\mathbf{Af}$:
The $i$-objects and $j$-objects form the
points and lines of a combinatorial affine plane over $\mathbb{F}_1$: one point incident with two lines which are incident with only that point.




There is also an odd-one-out class of examples which enters the picture.

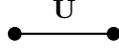
\begin{tikzpicture}[style=thick, scale=1.3]
\foreach \x in {1,2}{
\fill (\x,0) circle (2pt);}

\draw (1,0) -- (2,0);
\draw (1.5,.25) node {$\mathbf{U}$} ;

\end{tikzpicture}\hspace{0.5cm} $\mathbf{U}$\index{$\mathbf{U}$-label}:
The $i$-objects and $j$-objects form the
points and lines of a connected tree of valency $\leq 2$, with at least one end point. (Lines with one point are allowed, so at the ends, one can have end points or end lines.)

The unique examples of $\mathbf{Di}_2$, $\mathbf{A}_2$, $\mathbf{B}_2$ and $\mathbf{I}_2(m)$ are self-dual. The class described by $\mathbf{U}$ is also
self-dual, while 
the point-line dual of the $\mathbf{Af}$-type geometry is one of the rank $1$ examples. \\

\subsection{Cardinality}

By ``ordinary $\infty$-gons'' we mean connected trees with valency $2$ without end points. {\em The number of points is countable by the connectedness condition.}
The same is true for elements of type $\mathbf{U}$.\\

\medskip
\subsection{$\mathbb{F}_1$-Incidence geometries and base extension}

An incidence geometry which is {\em defined over} $\mathbb{F}_1$\index{incidence!geometry!over $\Fun$} could also be regarded as a couple $(S,\underline{S})$, where $S \in \mathbb{BT}_{\vert 1}$, $\underline{S} \in \mathbb{A}$, and 
$\underline{S} \cong \underline{\mA}(S)$. It is important to keep the category $\mathbb{S}$ in mind with objects the elements of 
$\underline{\mA}^{-1}(\underline{S})$ and natural morphisms. 

Many of the known fundamental finite incidence geometries (think in the first place of generalized polygons) come in ``classes''; for instance,
the $\mQ(4,k)$\index{$\mQ(4,k)$-functor} quadrangles can be seen as a functor which associates with each (possibly infinite) field $k$ the classical Moufang quadrangle $\mQ(4,k)$ \cite{POL}
 (in fact, it is defined as a $4$-dimensional hypersurface, so can also be regarded as a $\mathbb{Z}$-scheme).  It is convenient to consider the subcategory $\mathbb{FBT}_{\vert 1}$\index{$\mathbb{FBT}_{\vert 1}$} of $\mathbb{BT}_{\vert 1}$ which consists of those elements of $\mathbb{BT}_{\vert 1}$ which are members of infinite classes which arise as a functor from the category of fields to $\mathbb{BT}_{\vert 1}$. So to each $\Gamma \in \mathbb{FBT}_{\vert 1}$ we can associate at least one such 
 functor $F_{\Gamma}$ ($F_{\Gamma}$ need not be unique, as often over small finite fields such classes intersect in classical examples). 
 In $\mathbb{FBT}_{\vert 1}$ we can define a more refined version of $\mathbb{F}_1$-incidence geometry: it is a triple of the form
 \begin{equation}
 (F_{\Gamma},\Gamma,\underline{\Gamma}),
 \end{equation} 
 where $(\Gamma,\underline{\Gamma})$ is as above. In this context we call $F_{\Gamma}(k)$ a {\em $k$-extension}\index{$k$-extension} of $\Gamma$.

\section{Basic absolute  Linear Algebra}

In this section we describe several aspects of absolute Linear Algebra, partially and loosely
following the Kapranov-Smirnov document \cite{KapranovUN}.  We will usually only consider finite or infinitely countable dimensional vector spaces;
in the second chapter of the author, detailed considerations will be made on dimensions of any cardinality.

\subsection{Structural setting and mantra}

As we want to see $\Fun$ as a field which is different from $\F_2$, one often depicts $\Fun$ as the set $\{0,1\}$ for which we only have the following 
operations:
\begin{equation}
0\cdot 1 = 0 = 0\cdot 0\ \ \mbox{and}\ \ 1\cdot 1 = 1. 
\end{equation}

So in absolute Linear Algebra we are not allowed to have addition of vectors and we have to define everything in terms of scalar multiplication. \\

\subsection{Field extensions of $\Fun$}

Formally, for each $m \in \mathbb{N}^\times$ we define the {\em field extension}\index{field extension} $\Fun^m$\index{$\Fun^m$} of $\Fun$ of degree $m$ as the set
$\{0\} \cup \mu_m$, where $\mu_m$ is the (multiplicatively written) cyclic group of order $m$, and $0$ is an absorbing element for the extended
multiplication to $\{0\} \cup \mu_m$.\\

\subsection{Vector spaces over $\mathbb{F}_1^{(n)}$}

At the level of $\mathbb{F}_1$ we cannot make a distinction between 
affine spaces and vector spaces (as a torsor, nothing happens), so
in the vein of the previous section, a {\em vector/affine space}\index{vector space over $\F_1^n$}{affine space over $\F_1^n$} over $\mathbb{F}_1^n$, $n \in \mathbb{N}^\times$, is a triple $V = (\mathbf{0},X,\mu_n)$, where $\mathbf{0}$
is a distinguished point and $X$ a set, and
where $\mu_n$ acts freely on $X$. Each $\mu_n$-orbit corresponds to a direction. If $n = 1$, we get the notion considered in the previous section.
If the dimension is countably infinite, $\mu_n$ may be replaced by $\mathbb{Z}, +$ (the infinite cyclic group). Another definition is needed when the dimension is larger | we will come back to this issue in due course.\\

\subsection{Basis}

A {\em basis}\index{basis} of the $d$-dimensional $\mathbb{F}_1^n$-vector space $V = (\mathbf{0},X,\mu_n)$ is a set of $d$ elements in $X$ which are 
two by two contained in different $\mu_n$-orbits (so it is a set of representatives of the $\mu_n$-action); here, formally, $X$ consists of $dn$ elements, and $\mu_n$ is the cyclic group with $n$ elements. (If $d$ is not finite one selects exactly one element in each $\mu_n$-orbit.)
If $n = 1$, we only have $d$ elements in $X$
(which expresses the fact that the $\mathbb{F}_1$-linear group indeed is the symmetric group) - as such we obtain the {\em absolute basis}\index{absolute!basis}.

Once a choice of a basis $\{b_i \vert i \in I\}$ has been made, any element $v$ of $V$ can be uniquely written as $b_j^{\alpha^u}$, for 
unique $j \in I$ and $\alpha^u \in \mu_n = \langle \alpha \rangle$. So we can also represent $v$ by a $d$-tuple with exactly one nonzero entry, namely $b_j^{\alpha^u}$ (in 
the $j$-th column).

\subsection{Dimension}

In the notation of above, the {\em dimension}\index{dimension of vector space} of $V$ is given by  $\mathrm{card}(V)/n = d$ (the number of $\mu_n$-orbits). 

\subsection{Field extension}

Let $V$ be a (not necessarily finite dimensional) $d$-space over $\mathbb{F}_1^n$, $n$ finite, so that $\vert X = X_V \vert = dn$. For any positive integral divisor $m$ of $n$, with $n = mr$,
$V$ can also be seen as a $dr$-space over $\mathbb{F}_1^{m}$. Note that there is a unique cyclic subgroup $\mu_m$ of $\mu_n$ of size $m$, so 
there is only one way to do it (since we have  to preserve the structure of $V$ in the process).\\

In terms of affine spaces, interpretation over a subfield can be depicted as follows.\\

\begin{equation}
\begin{array}{ccc}
\AG(d,\mathbb{F}_1^n) &\longrightarrow &\AG(dr,\mathbb{F}_1^{m})\\
&&\\
\downarrow& &\downarrow\\
&&\\
(X,\mu_n)& \longrightarrow& (X,\mu_m)
\end{array}
\end{equation}

\subsection{Projective completion}

By definition, the {\em projective completion}\index{projective!completion} of a combinatorial affine space $\AG(n,\mathbb{K})$, $n \in \mathbb{N}$ and $\mathbb{K}$ a field, is the projective space $\PG(n,\mathbb{K})$ of the same dimension and defined over the same field, which one obtains by adding a ``hyperplane at infinity". The latter is a projective spave of one dimension less of which the subspaces represent parallel classes of subspaces of $\AG(n,\mathbb{K})$. For example, if $n = 2$ and $\mathbb{K} = \mathbb{R}$, we add a line at infinity which consists of parallel classes of affine lines.

Following the aforementioned considerations on projective completion, we immediately have the details on field extension for projective $\mathbb{F}_1^n$-spaces: starting from a projective $\mathbb{F}_1^n$-space $\mathbf{P} = \PG(d,\mathbb{F}_1^n)$, we choose an arbitrary point $x$, construct
the affine space $\mathbf{P} \setminus \{x\}$, blow up as above, and then projectively complete.\\

From the motivic point of view (which will be considered in the second chapter of the author in this volume), projective completion is extremely important: we refer the reader to the 
aforementioned chapter for the meaning of the mysterious identity

\begin{equation}
[\mathbf{P}^n(k)] = \mathbf{1} + \mathbb{L} + \mathbb{L}^2 + \cdots + \mathbb{L}^n.
\end{equation}

\subsection{Direct sums}

One defines a {\em direct sum}\index{direct sum} of (not necessarily finite dimensional) vector spaces $V$ and $W$, both defined over $\mathbb{F}_1^n$, as 
\begin{equation}
V \oplus W := V \coprod W,
\end{equation}
where the distinguished points $\mathbf{0}_V$ and $\mathbf{0}_W$ are identified.

\begin{theorem}[Dimension Theorem]
We have that 
\begin{equation}
\mathrm{dim}(V \oplus W) =  \mathrm{dim}(V) + \mathrm{dim}(W).
\end{equation}
\end{theorem}

\medskip
\subsection{Tensor products}

For defining the {\em tensor product}\index{tensor product} we start with vector spaces $V$ and $W$ defined over $\mathbb{F}_1^n$ and put
\begin{equation}
V \otimes W := V\times W,
\end{equation}
 the vector space corresponding to the Cartesian product of free $\mu_n$-sets. Here, we identify $\mathbf{0}_V \times W$ with 
 $V \times \mathbf{0}_W$.

 If the dimensions of $V$ and $W$ are respectively $d$ and $e$, then $V \otimes W$ consists of $den^2$ elements, so is of dimension $den$ over $\mathbb{F}_1^n$. In order to have a sensible notion of tensor product we have to eliminate the $n$-factor. We do this by identifying  $(x,y)$ with $(x^\nu,y^{\nu^{-1}})$ for any $\nu$ in $\mu_n$ and call the corresponding vector space $V \otimes W$ (so the latter is the quotient of $V \times W$ by the anti-diagonal action of $\mu_n$). If we denote the image of  $(x,y)$ in $V \otimes W$ by $x \otimes y$, then the identification merely says we can pull the $\mu_n$-action through the tensor sign:

 \begin{equation}
(x \otimes y)^{\nu} = x^{\nu} \otimes y = x \otimes y^{\nu}, 
 \end{equation}
 with $\nu \in \mu_n$ arbitrary. The set $V \otimes W$ is equipped with this $\mu_n$-action.

 \begin{theorem}[Dimension for tensor product]
 We have that 
 \begin{equation}
 \mathrm{dim}(V \otimes W) = \mathrm{dim}(V)\mathrm{dim}(W).
 \end{equation}
\end{theorem}

\subsection{Linear automorphisms}

A {\em linear automorphism}\index{linear!automorphism} $\alpha$ of an $\mathbb{F}_1^n$-vectorspace $V$ with basis $\{b_i\}$  is of the form 
\begin{equation}
\alpha(b_i) = b_{\sigma(i)}^{\beta_i} 
\end{equation}
for some power $\beta_i$ of the primitive $n$-th root of unity $\alpha$, and some permutation $\sigma \in \mathbf{S}_d$.
Then we have that 

\begin{equation}
\GL_d(\mathbb{F}_1^n) \cong \mathbf{S}_d \wr (\mu_n)^d.
\end{equation}

Elements of $\GL_d(\mathbb{F}_1^n)$ can be written as $d \times d$-matrices with precisely one element of $\mu_n$ in each row or column (and
conversely, any such element determines an element of $\GL_d(\mathbb{F}_1^n)$). 

(Note that the reason that rows and columns have only one nonzero element is that we do not have addition in our vector space.)

(In this setting, $\mathbf{S}_d$ is represented by $d\times d$-matrices with in each row and column exactly one $1$ | permutation matrices.)

\subsection{Determinants}

Using the setting of the precious paragraph, we define the {\em determinant}\index{determinant} 
\begin{equation}
\mathrm{det}(A) = \prod \beta_i \in \mu_n. 
\end{equation}

One verifies that the determinant is multiplicative and independent of the choice of basis.\\

\subsubsection{Examples}

Scalar multiplication by $\nu \in \mu_n$ gives an automorphism on any $d$-dimensional $\mathbb{F}_1^n$-vectorspace $V$ and the corresponding determinant clearly $\nu^d$. That is, the det-functor ``remembers'' the dimension modulo $n$. These mod $n$ features are a recurrent theme in absolute Linear Algebra.

Another example, which will become relevant when we come to reciprocity laws, is the following.

Take $n=2$. Then, an $\mathbb{F}_1^2$-vector space $V$ of dimension $d$ is a set $V$ consisting of $2d$ elements equipped with a free involution. Any linear automorphism $\alpha$ is represented by a $d \times d$-matrix $A$ having one nonzero entry in every row and column being equal to $+1$ or $-1$. Hence, the determinant $\mathrm{det}(A)$ is in $\{+1,-1\}$ (having put $\alpha = -1$).

On the other hand, by definition, the linear automorphism $\alpha$ determines a permutation on the $2d$ non-zero elements of $V$ (the elements of $X_V$). 
In fact, it is a permutation on the $d$ $\mu_2$-orbits.
The connection between these two interpretations is that $\mathrm{det}(A) = \mathrm{sgn}(A)$; the determinant gives the sign of the permutation.

\subsection{Power residue symbol}

For a prime power $q = p^k$ with $q \equiv 1 \mod{n}$, the roots of unity $\mu_n$ are contained in $\mathbb{F}_q^\times$, so that $\mathbb{F}_q$ is a vectorspace over $\mathbb{F}_1^n$. For any field unit $a \in \mathbb{F}_q^\times$ we have the power residue symbol

\begin{equation}
\left(\begin{array}{c}
a\\
\mathbb{F}_q
\end{array}\right)_n
= a^{\frac{q - 1}{n}} \in \mu_n.
\end{equation}

On the other hand, multiplication by $a$ is a linear automorphism $A$ on the $\mathbb{F}_1^n$-vectorspace $\mathbb{F}_q$ and hence we can look at its determinant $\mathrm{det}(A)$. The characteristic one interpretation of a classical lemma by Gauss now asserts: 

\begin{theorem}
The power residue symbol equals $\mathrm{det}(A)$.
\end{theorem}

\newpage
\section{Group representations}

In a recent communication \cite{JLP}, Javier L\'{o}pez Pe\~{n}a  observed that the classical
set-theoretical approach to representation theory of groups can be seen as a degenerate case of the general theory of linear representations through elementary $\Fun$-theory,  giving some common ground that explains the similarities between these two theories. We will start this section by describing this observation. Additional comments on (linear representations of) braid groups are made, some of which are taken from \cite{KapranovUN}.\\

\subsection{Linear representations}

Let us first recall that one may think of two basic ways of representing groups $G$: 

\begin{itemize}
\item[(a)]
The first one | with a more geometrical flavor | is looking for a  $\K$-vector space $V$, where $\K$ is a field,  and trying to describe $G$ inside the group  of automorphisms of $V$ by looking for group homomorphisms
\begin{equation}
\rho: G \longrightarrow \mathbf{GL}(V), 
\end{equation}
also called {\em linear representations}\index{linear!representations}. Note that one does not ask $\rho$ to be injective (although this property is often highly desirable), so that we could have a nontrivial kernel. If the representation, however, {\em is} faithful, the group $G$ is called {\em linear}\index{linear!group}.
\item[(b)]
The other | set-theoretical | one, consists of looking for sets $X$  endowed with a group action $G \curvearrowright X$ of $G$, also called  {\em $G$-sets}\index{$G$-set}, that allow us to describe $G$ as a group of permutations. So we seek for group homomorphisms
\begin{equation}
\gamma: G \longrightarrow \mathbf{Sym}(X).
\end{equation}
(Here, again, initially one does not ask $\gamma$ to be faithful, so that it could be that we describe $G/\mathrm{ker}(\gamma)$, rather than $G$, as a group inside $\mathbf{Sym}(X)$.)\\
\end{itemize}

\subsection{Representations over $\Fun$}

There are reasons to think that both approaches should have similar properties | after all we are trying to describe the same object in different guises. When one is looking for linear representations of a group, one has to fix the field over which the vector space is defined, as well as the dimension of the vector space. In particular, we might take representations defined over {\em finite fields}, diving into what is called ``modular representation theory''. 
But now we might try to study linear representations over $\Fun$. As $\Fun$-vector spaces  are just pointed sets $V = (\mathbf{0},\Omega)$ of which the cardinality $\omega = \vert \Omega \vert$ corresponds to the vector space dimension, and as the Tits argument on Chevalley groups tells us that the automorphism group of $V$ is the symmetric group $\mathbf{S}(\Omega)$,  we conclude that: 

\begin{proposition}[J.  L\'{o}pez Pe\~{n}a \cite{JLP}]
Linear representations  over $\Fun$  of a group are precisely permutation representations.\\
\end{proposition}

\subsection{Special example}

Consider a faithful linear representation
\begin{equation}
\rho: G \longrightarrow \mathbf{GL}(V)
\end{equation}
of some group $G$. 
In the author's second chapter we will encounter a particular kind of such a representation that will be very important for $\Fun$-geometry. It is defined by the property that the projection of $\rho(G)$ on $\PGL(V)$ (after dividing out the scalars) acts sharply transitively on the points of the corresponding projective space $\hP(V) = (V\setminus \{\mathbf{0}\})/\sim$. In fact, we will consider ``semi-linear representations'' (allowing twists by field automorphisms)
\begin{equation}
\rho: G \longrightarrow \mathbf{\Gamma L}(V)
\end{equation}
with the same properties. These representations  (called {\em Singer representations}\index{Singer representation}) will be used in a framework to understand the ad\`{e}le class space of a global field in characteristic $0$. \\

Note that the $\Fun$-analogs of such representations are nothing else than sharply transitive goup actions  $G \curvearrowright X$.\\

\subsection{Braid groups}

We introduce braid groups in three different ways.

\subsubsection{Braid groups via strings}

Let $n \in \mathbb{N}^\times$.
An {\em $n$-braid}\index{$n$-braid} consists of $n$ {\em strings}\index{string} or {\em strands}\index{strand} which connect $n$ ``top inputs'' (called $1,2,\ldots,n$) to $n$ ``bottom inputs'' (also called $1,2,\ldots,n$). Strands must move from top to bottom at any time.
Note that there is a natural composition of $n$-braids which also yields $n$-braids, and which makes the set of all $n$-braids into a group $\mathbb{B}_n$\index{$\mathbb{B}_n$} (taken that we identify braids which can be naturally transformed into each other | see \S\S \ref{braidfund} for more on these identifications).  

\begin{proposition}
The braid group $\mathbb{B}_n$ is torsion-free.
\end{proposition}

It is clear that an $n$-braid naturally induces an element of $\mathbf{S}_n$ (by going from top to bottom), and this association yields a surjective 
group homomorphism
\begin{equation}
\label{braideq}
\gamma: \mathbb{B}_n \longrightarrow \mathbf{S}_n.
\end{equation}

The kernel $\mathbb{P}_n$ of $\gamma$ is the {\em pure braid group}\index{pure braid group} (on $n$ strings), and consists of all $n$-braids with the same input and output position:

\begin{equation}
\1 \longrightarrow \mathbb{P}_n \longrightarrow \mathbb{B}_n \longrightarrow \mathbf{S}_n \longrightarrow \1.
\end{equation}

Any $n$-braid can be devided in intervals such that in each interval there is precisely one crossing of strings, so the set 
\begin{equation}
\{ \sigma_i \vert i \in \{1,2,\ldots,n - 1\}\},
\end{equation}
where $\sigma_j$ is defined as the $n$-braid in which there is an overcrossing between inputs $j$ and $j + 1$ (and nothing else), generates $\mathbb{B}_n$. (Undercrossing yields the inverses of the  generators.)

\begin{figure}
\centering
\begin{tikzpicture}[style=thick, scale=1.5]
\foreach \x in {1,2,3,4,5}{
\fill (\x,0) circle (2pt);}
\foreach \x in {1,2,3,4,5}{
\fill (\x,-3) circle (2pt);}
\draw (1,0.3) node {1};
\draw (2,0.3) node {2};
\draw (3,0.3) node {3};
\draw (4,0.3) node {4};
\draw (5,0.3) node {5};
\draw (1,-3.3) node {1};
\draw (2,-3.3) node {2};
\draw (3,-3.3) node {3};
\draw (4,-3.3) node {4};
\draw (5,-3.3) node {5};
\draw (1,0) -- (1,-3);
\draw (2,0) -- (2,-3);
\draw (3,0) -- (3,-3);
\draw (4,0) -- (5,-3);
\draw (5,0) -- (4.6,-1.4);
\draw (4,-3) -- (4.45,-1.8);
\end{tikzpicture} 
\caption{The generator $\sigma_4$ in $\mathbb{B}_5$.}
\end{figure}
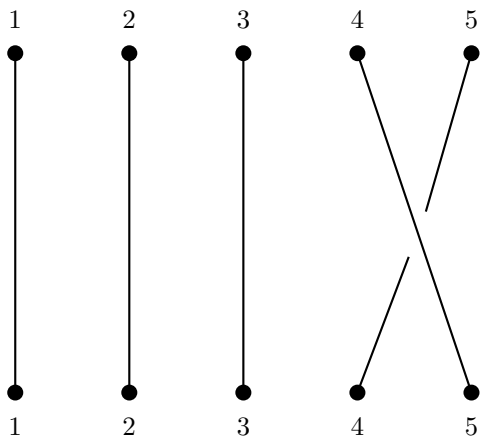

Now in terms of relations in the $\sigma_j$s, Artin proved that all relations can be deduced from only two, namely:
\begin{equation}
\left\{ 
\begin{array}{cccc}
 \sigma_i\sigma_j &= &\sigma_j\sigma_i &\mbox{if}\ \ \vert i - j\vert \geq 2\\
\sigma_i\sigma_{i + 1}\sigma_i &= &\sigma_{i + 1}\sigma_i\sigma_{i + 1} & 
\end{array}
\right.
\end{equation}

So
\begin{equation}
\mathbb{B}_n \cong  \langle \sigma_i, i = 1,2,\ldots,n - 1 \vert  \sigma_i\sigma_j = \sigma_j\sigma_i\ \mbox{if}\  \vert i - j\vert \geq 2, 
\sigma_i\sigma_{i + 1}\sigma_i = \sigma_{i + 1}\sigma_i\sigma_{i + 1}   \rangle.
\end{equation}

Recalling the presentation by generators and relations of $\mathbf{S}_n$ as a Coxeter group, we now have an explicit form for the homomorphism $\gamma$ (noting that if the extra relations $\{ {\sigma_i}^2 = \mathrm{id} \vert i = 1,2,\ldots,n - 1\}$ would be added, we would get the symmetric group).\\

\subsubsection{Braid groups as fundamental groups}
\label{braidfund}

Let $X$ be a connected topological space, let $d \geq 2$ be a positive integer, and consider the $d$-fold Cartesian product of $d$ copies of $X$, denoted by $X^d$. 
Let $\widetilde{X^d}$ be the {\em symmetrized $d$-fold Cartesian product}\index{symmetrized Cartesian product}, which is defined by moding out the natural action of the symmetric group $\mathbf{S}_d$ on the indices of the Cartesian coordinates. (That is, we consider  unordered $d$-tuples.) We only want to consider elements with no repeated entries, so we take out the ``hyperplanes'' with equations $x_i = x_j$ on the coordinates ($i \ne j$).  The obtained space is denoted by $\widetilde{X^d}_*$\index{$\widetilde{X^d}_*$} (and its elements can be identified with the subsets of $X$ of size $d$). We define the {\em braid group}\index{braid group} $\mathbb{B}_d(X)$\index{$\mathbb{B}_d(X)$} of $X$ on $d$ {\em strings}\index{string}  as the fundamental group of this space with respect to an arbitrary point $x_0$ (of which the choice does not affect the isomorphism class of the group):
\begin{equation}
\mathbb{B}_d(X) := \pi_1(\widetilde{X^d}_*,x_0).
\end{equation}

Now put $X = \mathbb{C}$; then there is a natural isomorphism between $\mathbb{B}_d(\mathbb{C})$ and $M\mathbb{C}^d[X]$\index{$M\mathbb{C}^d[X]$}, which is the set of polynomials in $X$ over $\mathbb{C}$ of degree $d$ and with leading coefficient $1$, without multiple roots. The map is given by 
\begin{equation}
\{u_1,\ldots,u_n\} \in  \mathbb{B}_d(\mathbb{C})\ \  \longrightarrow\ \ (X - u_1)\cdots(X - u_n) \in M\mathbb{C}^d[X]. 
\end{equation}

One can show that $\mathbb{B}_n(\mathbb{C})$ is isomorphic to the group $\mathbb{B}_n$ as above.\\

\medskip
\subsubsection{Braid groups via graphs of type $\mathbf{A}_{n - 1}$}

Let $\Gamma = (V,E)$ be a graph, with vertex set $V$ and edge set $E$.  We define the {\em Artin group}\index{Artin group} $A(\Gamma)$ as 
the free group $F(V)$ generated by the elements of $V$, modulo the following relations:
\begin{itemize}
\item[(R1)]
If $x$ and $y$ are adjacent vertices, then
\begin{equation}
xyx = yxy.
\end{equation}
\item[(R2)]
If $x$ and $y$ are not adjacent, they commute.
\end{itemize}

We also say that $A(\Gamma)$ is an Artin group ``of type $\Gamma$''\index{Artin group!of type $\Gamma$}. If $\Gamma$ is a Coxeter graph of type $\mathbf{A}_{n - 1}$, then $A(\Gamma)$ is isomorphic to $\mathbb{B}_n$.

\bigskip
$\mathbf{A}_{n - 1}$: \begin{tikzpicture}[style=thick, scale=1.5]
\foreach \x in {1,2,3,5,6}{
\fill (\x,0) circle (2pt);}

\draw (1,0) -- (2,0);
\draw (2,0) -- (3,0);
\draw (3,0) -- (3.5,0);
\draw (4.5,0) -- (5,0);
\draw (5,0) -- (6,0);
\draw (4,0) node {$\dots$} ;

\end{tikzpicture}

\medskip
Let $\Gamma$ be a graph,
and let $d \in \mathbb{N}^\times$. The {\em Shepard group}\index{Shepard group} $A(\Gamma,d)$ is the quotient of $A(\Gamma)$ by the relations
$v^d = \1$ for all $v \in V$.

\begin{proposition}
Let $\Gamma$ be a graph, and $A(\Gamma)$ its Artin group. Then $A(\Gamma,2)$ is the Coxeter group related to $\Gamma$. If $\Gamma$ is a Coxeter graph of type $\mathbf{A}_{n - 1}$, then $A(\Gamma) \cong \mathbb{B}_n$ and $A(\Gamma,2) \cong \mathbf{S}_n$.
\end{proposition}

\medskip
\subsubsection{Linear representations of $\mathbb{B}_n$}

Since we know that the symmetric groups are general linear groups over $\Fun$, one might wonder (and this is wat Kapranov and Smirnov do in their manuscript \cite{KapranovUN}), whether the expression (\ref{braideq}) fits into a diagram

\begin{equation}
\begin{array}{ccc}
?_1 &\longrightarrow &\GL_n(\mathbb{F}_q)\\
&&\\
\downarrow& &\downarrow\\
&&\\
 ?_2 \cong \mathbb{B}_n &\overset{\gamma}\longrightarrow &\GL_n(\Fun) \cong \mathbf{S}_n,\\
&&\\
\end{array}
\end{equation}
where passing from the first row to the second means passing to the limit $q \longrightarrow 1$. The first row should be seen as a class of arrows (with $q$ taking values in the set of prime powers).

Kapranov and Smirnov  suggest to replace $?_2$ by $\GL_n(\Fun[X])$, and also suggest that the evaluation morphism $X = 0$ yields $\gamma$. Their motivation is a theorem of Drinfeld which states that over a finite field $\F_q$, the profinite completion of $\GL_n(\F_q[X])$ is embedded in the fundamental group of the space of $q$-polynomials of degree $n$ in a rather similar way $\mathbb{B}_n$ is the fundamental group of $M\mathbb{C}^n[X]$. Still, as we will see in the second chapter of the author in this monograph, in the direction we want to take expressions such as $\Fun[X]$, this idea does not make much sense, as $\GL_n(\Fun[X])$ will not be a group. (Another rather natural candidate would be $\GL_n(\Fun[X,X^{-1}])$, but $\mathbf{S}_n$ is a subgroup, while $\mathbb{B}_n$ is torsion-free.)

In any case, an $\Fun^n$-linear representation of $\mathbb{B}_n$ is given by the map
\begin{equation}
\rho: \sigma_i \longrightarrow
\left( 
\begin{array}{cccc}
\I_{i  - 1} &&& \\
&0 &\mu &\\
&\mu^{-1} &0 &\\
&&&\I_{n - 1 - i}
\end{array}
\right),
\end{equation}
where $\Fun^n = \mu_n \cup \{0\} = \langle \mu \rangle \cup \{0\}$. This representation is of course not faithful, and  inside $\GL_n(\Fun^n)$ the elements
$\rho(\sigma_i)$, $i = 1,2,\ldots,n - 1$ generate a subgroup isomorphic to $\mathbf{S}_n$.\\

The linearity of braid groups over ``real fields'' was only quite recently obtained by Bigelow \cite{Bigelow} and Krammer \cite{Krammer} (independently), and presented the solution of a major open problem. Note that any faithful linear representation
\begin{equation}
\rho: \mathbb{B}_n \longrightarrow \GL_m(R)
\end{equation}
with $m \in \mathbb{N}^\times$ and $R$ an ``$\Fun$-ring'' (see the author's second chapter) which is embeddable in a field (or division ring) $\K$ would give a faithful linear representation over $\K$.

\newpage
\section{From absolute mantra to absolute Algebraic Geometry}

In the early nineties, Christopher Deninger published his studies (\cite{Deninger1991}, \cite{Deninger1992}, \cite{Deninger1994}) on motives and regularized determinants. In \cite{Deninger1992}, Deninger gave a description of conditions on a category of motives that would admit a translation of Weil's proof of the Riemann Hypothesis for function fields of projective curves over finite fields $\F_q$ to the hypothetical curve $\overline{\Spec(\Z)}$. In particular, he showed that the following formula would hold:

\begin{equation}
\zeta_{\overline{\Spec(\Z)}}(s) = 2^{-1/2}\pi^{-s/2}\Gamma(\frac s2)\zeta(s)  
		=  \frac{\Rprod_\rho\frac{s - \rho}{2\pi}}{\frac{s}{2\pi}\frac{s - 1}{2\pi}} \overset{?}{=} \nonumber \\
\end{equation}
	\begin{eqnarray} 	
	 \frac{\mbox{\textsc{Det}}\Bigl(\frac 1{2\pi}(s\cdot\id - \Theta)\Bigl| H^1(\overline{\Spec(\Z)},*_{\mathrm{abs}})\Bigr.\Bigr)}{\mbox{\textsc{Det}}\Bigl(\frac 1{2\pi}(s\cdot\id -\Theta)\Bigl| H^0(\overline{\Spec(\Z)},*_{\mathrm{abs}})\Bigr.\Bigr)\mbox{\textsc{Det}}\Bigl(\frac 1{2\pi}(s\cdot\id - \Theta)\Bigl| H^2(\overline{\Spec(\Z)},*_{\mathrm{abs}})\Bigr.\Bigr)}, 
	\end{eqnarray}
where $\Rprod$ is the infinite {\em regularized product}, similarly
$\mbox{\textsc{Det}}$ denotes the {\em regularized determinant}\index{regularized determinant} (cf. the author's second chapter), $\Theta$ is an ``absolute'' Frobenius endomorphism and the $H^i(\overline{\Spec(\Z)},*_{\mathrm{abs}})$ are certain proposed cohomology groups. The $\rho$s run through the set of critical zeroes of the classical Riemann zeta. 

This description combines with Kurokawa's work on multiple zeta functions (\cite{Kurokawa1992}) from 1992 to  the hope that there are motives $h^0$ (``the absolute point''), $h^1$ and $h^2$ (``the absolute Lefschetz motive'')\index{absolute!Lefschetz motive} with zeta functions
	\begin{equation}
	\label{eqzeta}
		\zeta_{h^w}(s) \ = \ \mbox{\textsc{Det}}\Bigl(\frac 1{2\pi}(s\cdot\id-\Theta)\Bigl| H^w(\overline{\Spec(\Z)},*_{\mathrm{abs}})\Bigr.\Bigr) 
	\end{equation}
for $w=0,1,2$. Deninger computed that $\zeta_{h^0}(s)=s/2\pi$ and $\zeta_{h^2}(s)=(s-1)/2\pi$. Manin proposed in \cite{Manin} the interpretation of $h^0$ as $\Spec(\Fun)$ and the interpretation of $h^2$ as the affine line over $\Fun$. The search for a proof of the Riemann Hypothesis became a main motivation to look for a geometric theory over $\Fun$. \\

About ten years after Manin's lecture notes \cite{Manin}, the first papers got published in which scheme theories over $\mathbb{F}_1$ were developed, the first one being Deitmar's important paper \cite{Deitmarschemes2} in 2005 (cf. the author's second chapter for more details). One year before, Soul\'{e} already published his $\mathbb{F}_1$-approach to varieties \cite{Soule}.

We have seen that once we forget about addition, a good basic theory of Linear Algebra can be developed which agrees with 
Tits's initial observations on the symmetric groups (and their geometries) as limit objects over $\Fun$. 
In the author's second chapter, several versions of Algebraic Geometry will be described in detail. And in yet another chapter, Lorscheid explains his side of the story with much rigor. 
Those will all be based on the fundamental observation that if we want develop a scheme theory over $\Fun$, we need algebraic objects in which we do not have addition at hand. One of the aims is to have a construction of base extension
\begin{equation}
\Fun \longrightarrow \Z
\end{equation}
in order to be able to pass to Grothendieck's $\Z$-schemes from below.\\

We will revisit the combinatorial realizations
of the obtained scheme theories again and again, and show that they are in perfect harmony with what was obtained in the present chapter.


\newpage
\printindex

%
%
%





\end{document}